%% file: Main_SIIMS_Manuscript.tex
\documentclass[onefignum,onetabnum]{siamonline190516}

\input{ex_shared}

\ifpdf
\hypersetup{
  pdftitle={ImaSK: Stochastic Multiresolution Image Domain Sketching for Inverse Imaging Problems},
  pdfauthor={A. Perelli, C.-B. Sch{\"o}nlieb, M. J. Ehrhardt}
}
\fi

\pgfplotsset{compat=1.18} 

\begin{document}

\maketitle

\input{Abstract}

\input{Sec1_Introduction}

\input{Sec2_ImaSk_Algorithm}

\input{Sec3_Convergence_Analysis}

\input{Sec4_Simulations_Results}

\bibliographystyle{siamplain}
\bibliography{references}

\end{document}

%% file: ex_shared.tex
\usepackage{lipsum}
\usepackage{amsfonts}
\usepackage{amssymb}
\usepackage{graphicx}
\usepackage[caption=false]{subfig} %

\usepackage{capt-of}

\usepackage{epstopdf}
\usepackage{bbm}

\usepackage{enumitem}
\setlist[enumerate]{leftmargin=.5in}
\setlist[itemize]{leftmargin=.5in}

\usepackage{csquotes}

\newsiamremark{remark}{Remark}
\newsiamremark{hypothesis}{Hypothesis}
\crefname{hypothesis}{Hypothesis}{Hypotheses}
\newsiamthm{claim}{Claim}
\newsiamthm{assumption}{Assumption}

\def\imask{\texttt{ImaSk}}
\def\imaskseq{\texttt{ImaSk-Seq}}

\usepackage{pgfplots}
\usepgfplotslibrary{fillbetween}

\usepackage{tikz}

\usetikzlibrary{spy}
\usetikzlibrary{calc}

\newcommand{\specialcell}[2][c]{%
	\begin{tabular}[#1]{@{}c@{}}#2\end{tabular}}

\usepackage{chngcntr}
\counterwithout{equation}{section}

\tikzstyle{only in spy node magn 1.75}=[%
transform canvas={%
	shift={(tikzspyinnode)},
	scale=1.75,
}
]

\usepackage{adjustbox}

\usepackage{pgf}

\usetikzlibrary{matrix}
\usetikzlibrary{decorations.pathmorphing} 
\usetikzlibrary{fit}					
\usetikzlibrary{backgrounds}	
\usetikzlibrary{shapes.arrows}
\usetikzlibrary{shapes.geometric}
\usetikzlibrary{shapes.multipart}
\usetikzlibrary{matrix,arrows,decorations.pathreplacing,positioning}
\tikzstyle{rec}=[draw,rectangle, minimum height=2cm]
\tikzset{>=stealth', punkt/.style={rectangle, 
		fill=gray!40, 
		draw=black, very thick, node distance=45pt, text width=4.5em, minimum height=3em, text centered}}
\tikzset{>=stealth', punkt_l/.style={rectangle, 
		fill=gray!40, 
		draw=black, very thick, text width=3em, minimum height=2.5em, text centered}}

\tikzstyle{background}=[rectangle,fill=green!20,inner sep=0.1cm,rounded corners=4mm,minimum height=3em]
\tikzstyle{sum} = [draw, fill=gray!40, circle, node distance=1cm]
\tikzstyle{dot} = [circle, fill=black, inner sep=0pt, minimum size=5pt, node contents={}]
\tikzstyle{text_n} = [node distance=30pt, text width=8em, minimum height=2.5em, text centered]
\tikzstyle{text_upn} = [node distance=0pt, text width=12em, minimum height=2.5em, text centered]
\tikzstyle{fig_n} = [node distance=0pt, inner sep=0cm]

\def\*#1{\mathbf{#1}}
\def\+#1{\mathcal{#1}}
\def\-#1{\mathbb{#1}}
\def\R{\mathbb{R}}
\def\RI{\mathbb{R}_\infty}
\def\XX{\mathcal{X}}
\def\YY{\mathcal{Y}}
\def\ZZ{\mathcal{Z}}
\def\N{\mathbb{N}}

\def\Expect{\mathbb{E}}

\renewcommand{\tilde}{\widetilde}

\usepackage{algorithm,algcompatible}
\usepackage{bm}

\algnewcommand\INPUT{\item[\textbf{Input:}]}%
\algnewcommand\OUTPUT{\item[\textbf{Output:}]}%

\def\it{k}
\def\It{K}

\headers{ImaSk: Stochastic Multiresolution Image Domain Sketching}{}

\title{Stochastic Multiresolution Image Sketching for Inverse Imaging Problems\thanks{Submitted to the editors of SIIMS.
\funding{
The authors acknowledge support from the Royal Academy of Engineering under the RAEng / Leverhulme Trust Research Fellowships programme (award LTRF-2324-20-160), the Leverhulme Trust (ECF-2019-478), the Philip Leverhulme Prize, the Royal Society Wolfson Fellowship, the EPSRC advanced career fellowship EP/V029428/1, the EPSRC programme grant EP/V026259/1, and the EPSRC grants EP/S026045/1, EP/T026693/1, EP/T003553/1, EP/N014588/1, EP/T017961/1, the Wellcome Innovator Awards 215733/Z/19/Z and 221633/Z/20/Z, the European Union Horizon 2020 research and innovation programme under the Marie Skodowska-Curie grant agreement NoMADS and REMODEL, the Cantab Capital Institute for the Mathematics of Information and the Alan Turing Institute.
This research was also supported by the NIHR Cambridge Biomedical Research Centre (NIHR203312). The views expressed are those of the author(s) and not necessarily those of the NIHR or the Department of Health and Social Care.}}}

\author{Alessandro Perelli\thanks{School of Science and Engineering, University of Dundee, DD1 4HN, UK (\email{aperelli001@dundee.ac.uk}).}
\and Carola-Bibiane Sch{\"o}nlieb\thanks{Department of Applied Mathematics and Theoretical Physics, University of Cambridge, UK (\email{cbs31@cam.ac.uk}).}
\and Matthias J. Ehrhardt\thanks{Department of Mathematical Sciences, University of Bath, UK (\email{m.ehrhardt@bath.ac.uk})}}

\usepackage{amsopn}

%% file: Abstract.tex
\begin{abstract}
A challenge in high-dimensional inverse problems is developing iterative solvers to find the accurate solution of regularized optimization problems with low computational cost. An important example is computed tomography (CT) where both image and data sizes are large and therefore the forward model is costly to evaluate. Since several 
years algorithms from stochastic optimization are used for tomographic image reconstruction with great success by subsampling the data. Here we propose a novel way how stochastic optimization can be used to speed up image reconstruction by means of image domain sketching such that at each iteration an image of different resolution is being used. Hence, we coin this algorithm \imask. By considering an associated saddle-point problem, we can formulate \imask~as a gradient-based algorithm where the gradient is approximated in the same spirit as the stochastic average gradient amélioré (SAGA) and uses at each iteration one of these multiresolution operators at random. We prove that \imask~is linearly converging for linear forward models with strongly convex regularization functions. Numerical simulations on CT show that \imask~is effective and increasing the number of multiresolution operators reduces the computational time to reach the modeled solution. 	
\end{abstract}



%% file: Sec1_Introduction.tex
\section{Introduction}

Inverse problems involving estimating unknown images from noisy measurements arises in many medical and non-medical imaging applications such as X-ray Computed Tomography (CT), Positron Emission Tomography (PET) and Magnetic Resonance Imaging (MRI).
If the noise is modeled as Gaussian, then the problem can be mathematically described as $\*b = \*K\*x + \*n$ where $\*K\in\R^{m\times d}$ represents the physical forward model and $\*n\in\R^m$ the mean-zero Gaussian noise. Nowadays with the deluge of acquired data and the aim to reconstruct images at fine resolution, an emerging challenge is constituted by reducing the computational cost while preserving the accuracy of the solution \cite{arridge2019solving}.

An established framework for estimating the unknown high-dimensional vector $\*x$ from the measurements $\*b$ is to solve a model-based optimization problem of the form
\begin{equation}\label{eq:min_obj}
\min_{\*x\in\R^d} \frac12 \|\*K\*x - \*b\|^2 + R(\*x)
\end{equation}
where $R : \R^d \to \RI := \R \cup \{\infty\}$ describes the regularization term which enforce a particular structure on the possible solutions for $\*x$ \cite{fessler2010model}. 
Many iterative solvers have been developed to minimize the original model-based cost function \cite{kim2016}. 
First-order methods and its accelerated versions generally involve operations like matrix-vector multiplications and evaluations of proximal operators \cite{combettes2011proximal}, which are for many problems in closed form or easy to compute iteratively \cite{beck2009fastshrink}.
From a computational point of view, a main bottleneck of any iterative solvers used to minimize \cref{eq:min_obj} is the application of the forward operator $\*K$ and its adjoint \cite{ahmad2020plug, perelli2020compressive}.

A popular approach to overcome this computational bottleneck is to change the discretisation of the image. E.g.\ in \cite{Lesonen2024} chooses a nonuniform triangular grid for PET reconstruction where the resolution is selected a-priori from a second modality. An inexact multi-level version of FISTA is proposed in \cite{Lauga2023}.

A similar concept was previously exploited by multi-grid reconstruction methods which are based on partitioning the image of interest in order to reduce the computation time and memory required \cite{marlevi2019}, since in practical settings often only a specific region of interest needs to be analyzed at full resolution. Alternative methods are based on the multiresolution wavelet structure \cite{vonesch2009} which allows to iteratively impose regularization constraints on an hierarchy of fine to coarse resolution images.

Recently a wide area of research is constituted by the development of scalable gradient-based algorithms which use only a random subset of the measurements at each iteration; notable methods are stochastic gradient descent (SGD) \cite{kim2014} and other Stochastic Variance Reduced methods \cite{xiao2014, gower2020} such as the improved stochastic average gradient (SAGA) \cite{defazio2014} which can reduce the variance of the stochastic gradient estimator progressively and improve the convergence of SGD methods. Stochastic algorithms have also been applied to saddle-point optimization problems \cite{Balamurugan2016, Chambolle2011pdhg, ehrhardt2019faster, Devraj2019}. 

Furthermore, different types of dimensionality reduction operators, called sketches, in the measurement domain have been developed within iterative convergent algorithms \cite{woodruff2014sketching}. Random sketches are used in second-order method for optimization to perform approximate Newton step using a randomly projected or sub-sampled Hessian \cite{pilanci2017newton}. 

Randomized algorithms using sketched gradients are leveraged to solve least squares problems or to reduce the computation of proximal algorithms \cite{tang2017gradient} or optimization using regularization by denoising \cite{perelli2021regularization}. Over the last years, several of the aforementioned algorithms have been applied to imaging problems, most notably to CT and PET, see e.g.\ \cite{ehrhardt2019faster, Twyman2022, Tang2019} but also other imaging modalities like magnetic particle imaging \cite{Zdun2021}. For an overview on the field and further references, we refer to \cite{ehrhardt2024guide}. As in machine learning, all of these approaches randomly sample the data. In contrast, we use a novel random sampling in image space.

\subsection{Contribution}

Inspired by the notion of sketching, we propose a novel algorithm that uses cheaper matrix-vector products essentially by sketching in the image domain. To this end we decompose the forward operator using a family of multiresolution operators and consider the corresponding saddle point formulation. The latter reveals a structure commonly encountered in stochastic optimization which can be approached by various gradient estimators such as SAGA. Note that the original formulation does not the reveal the averaging property that is used for efficient stochastic gradient based solvers. From a computational perspective, this allows to reduce the per-iteration cost of the iterative algorithm and thereby potentially an overall more efficient algorithm. We analyze its linear convergence rate and give sufficient conditions on its parameters for convergence. While this approach is similar to \cite{Balamurugan2016}, our analysis reveals more general conditions and better convergence rates. We also study the behavior of the algorithm for CT image reconstruction in terms of the number of resolutions, step-size, convergence rate and computational time.

%% file: Sec2_ImaSk_Algorithm.tex
\section{\imask~Algorithm}\label{sec:ImaSk_algorithm}
In this section we present the proposed algorithm for image reconstruction, i.e., for solving \cref{eq:min_obj}. 

\subsection{The Idea in a Nutshell}
As outlined before, we aim to exploit sketching in the image domain. To this end, we define a family of operators $\{\*S_i\}$ which significantly make the evaluation of $\*K \*x$ computationally cheaper, i.e., $\*K \* S_i \*x$ is cheaper than $\*K \*x$, and does not introduce bias in the solution, i.e., \enquote{on average} they cancel another, 
\begin{align}
    \mathbb E_i \*S_i = \sum_{i=1}^r p_i \*S_i = \*I \label{eq:si},
\end{align}
and are similar to $\*K$, i.e., $\*K \*x \approx \*K \* S_i \*x$. An example family of operators $\{\*S_i\}$ is illustrated in \cref{fig:CTsketch_ASTRA_Idecim}. We see that the original image on the left is approximated by lower and lower resolutions. This is compensated by the image on the right which essentially is a sharpened version of the ground truth.
\begin{figure*}[!ht]
	\centering
	\small\addtolength{\tabcolsep}{-9pt}
	\renewcommand{\arraystretch}{0.1}
	\begin{tabular}{|c|c|c|c|c|}
		\multicolumn{1}{c}{Ground truth}
		& 
		\multicolumn{4}{c}{$\*S_i\*x^*$} \vspace{.15cm}\\
		
		\multicolumn{1}{c}{} & \multicolumn{4}{c}{\downbracefill} \vspace{.15cm}\\
		\hline
		
		\small{(a) $d = 512^2$} 
		& 
		\small{(b) $d_i = 256^2$}
		& 
		\small{(c) $d_i = 128^2$} 
		& 
		\small{(d) $d_i = 64^2$} 
		& 
		\small{(e) $d_i = 512^2$}  \\
		
		\begin{tikzpicture}
			\begin{scope}[spy using outlines={rectangle,yellow,magnification=3,size=12mm,connect spies}]   
    \node {\includegraphics[trim=30 60 30 60, clip, width=0.2\textwidth]{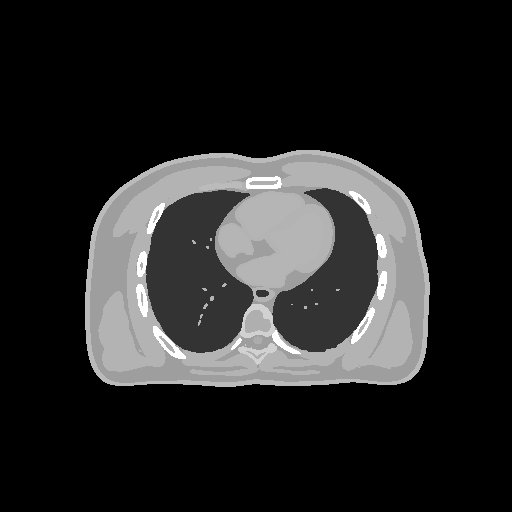}}; 
    \spy on (0,-.5) in node [left] at (-0.3,-1.8); 
    \end{scope}
    \end{tikzpicture} 
    &
    \begin{tikzpicture}
    \begin{scope}[spy using outlines={rectangle,yellow,magnification=3,size=12mm,connect spies}]   
    \node {\includegraphics[trim=30 60 30 60, clip, width=0.2\textwidth]{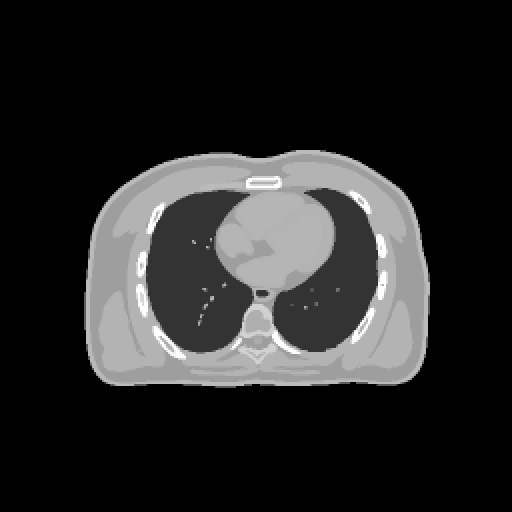}};
    \spy on (0,-.5) in node [left] at (-0.3,-1.8); 
    \end{scope}
    \end{tikzpicture} 
    &
    \begin{tikzpicture}
    \begin{scope}[spy using outlines={rectangle,yellow,magnification=3,size=12mm,connect spies}]   
    \node {\includegraphics[trim=30 60 30 60, clip, width=0.2\textwidth]{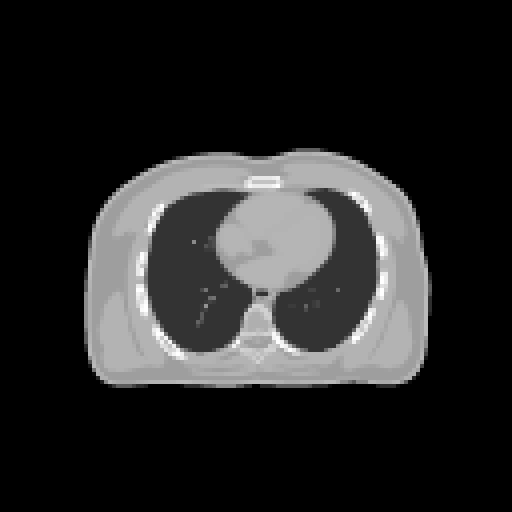}};
    \spy on (0,-.5) in node [left] at (-0.3,-1.8); 
    \end{scope}
    \end{tikzpicture}  
    &       
    \begin{tikzpicture}
    \begin{scope}[spy using outlines={rectangle,yellow,magnification=3,size=12mm,connect spies}]   
    \node {\includegraphics[trim=30 60 30 60, clip, width=0.2\textwidth]{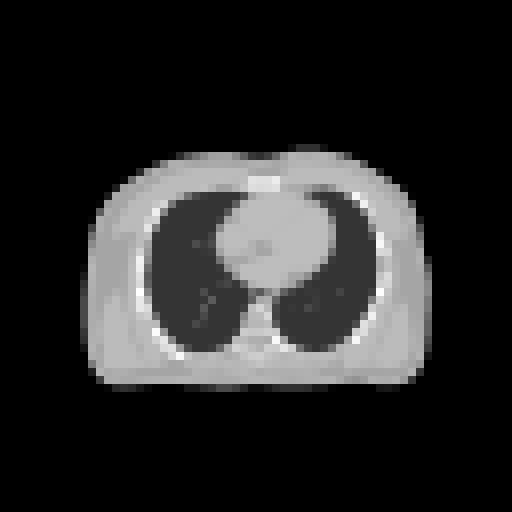}};
    \spy on (0,-.5) in node [left] at (-0.3,-1.8); 
    \end{scope}
    \end{tikzpicture} 
    &                             
    \begin{tikzpicture}
    \begin{scope}[spy using outlines={rectangle,yellow,magnification=3,size=12mm,connect spies}]
    \node {\includegraphics[trim=30 60 30 60, clip, width=0.2\textwidth]{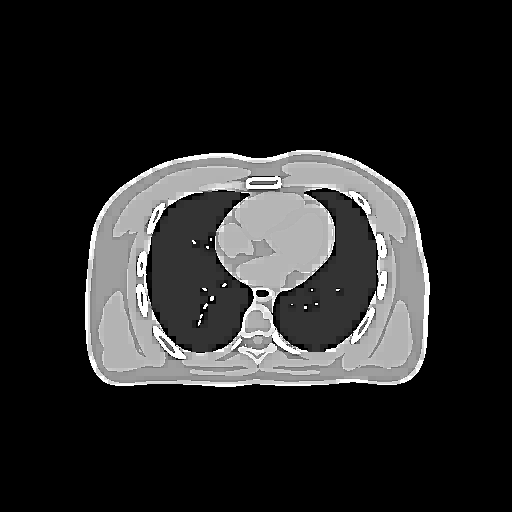}};
    \spy on (0,-.5) in node [left] at (-0.3,-1.8); 
    \end{scope}
    \end{tikzpicture}   
    \\  \hline
    \end{tabular}
    \caption{Illustration of multiresolution sketches. The sketches (here $r=4$) are averaging over small neighborhoods, see \cref{sec:sketches} for details. (a) Ground truth image $\*x^*$ at full resolution $d = 512^2$, (b, c, d) increasingly lower resolution mappings $\* S_i \*x^*$ and (e) resolution compensation. All images are shown on the same color range $[-0.1, 1.1]$.} \label{fig:CTsketch_ASTRA_Idecim}
\end{figure*}

Then, with $\*K_i := \*K \*S_i$, we aim to find solutions of \cref{eq:min_obj} via 
\begin{equation}\label{eq:min_obj_new}
\min_{\*x\in\R^d} \frac12 \left\|\sum_{i=1}^r p_i \*K_i\*x - \*b\right\|^2 + R(\*x)
\end{equation}
where, for computational efficiency, the algorithm is supposed to only evaluate one $\* K_i$ per iteration. Note that the structure of \cref{eq:min_obj_new} with the sum inside the nonlinear function is different from the usual template in stochastic optimization, see e.g., \cite{defazio2014, chambolle2018}. \imask~handles this by considering the associated saddle-point problem, see also \cref{sec:saddle-point}.

\subsection{The Algorithm}
Given the operators $\*K_i$, we propose to solve the optimization problem \cref{eq:min_obj} using the primal-dual algorithm \imask~whose pseudo-code is listed in \cref{table:ImaSk}.
\begin{algorithm}[!h]
    \caption{\imask~for CT image reconstruction, i.e., solve \cref{eq:min_obj}. Default values in brackets.} \label{table:ImaSk}
    \begin{algorithmic}[1]
    \REQUIRE number of iterations $\It$, probabilities $p = (p_i)_{i=1}^r \in (0,1]^r$, step size $\sigma$, strong convexity of regularizer $\mu$
    \INPUT initial iterate $\*x^0 (= \*0) \in \R^d$, $\*y^0 (= \* 0) \in \R^m$, \newline memory $\*\phi^{0} = (\*K_i\*x^0)_{i=1}^r (= \*0) \in (\R^{d})^r$, $\*\psi^{0} = (\*K_i^*\*y^0)_{i=1}^r (= \*0) \in (\R^m)^r$
    \OUTPUT $\*x^\It$
    \FOR{$k = 0, \ldots, \It-1$}
    \STATE Sample $i_\it\in \{1, \dots, r\}$ at random with probability $p$.
    \STATE $\*\phi^{\it+1}_i = \begin{cases} \*K^T_{i_\it} \*y^\it & \text{if $i = i_\it$} \\ \*\phi^\it_i & \text{else} \end{cases}$
    \STATE $\*\psi^{\it+1}_i = \begin{cases} \*K_{i_\it} \*x^\it & \text{if $i = i_\it$} \\ \*\psi^\it_i & \text{else} \end{cases}$
    \STATE $\*\xi^\it = \*\phi^{\it+1}_{i_\it} - \*\phi^\it_{i_\it} + \sum_{i=1}^r p_i \*\phi^\it_i$
    \STATE $\*\zeta^\it = \*\psi^{\it+1}_{i_\it} - \*\psi^\it_{i_\it} + \sum_{i=1}^r p_i \*\psi^\it_i$
    \STATE $\*x^{\it+1} = \mathrm{prox}_{\sigma \mu^{-1} R}\,\left( \*x^\it - \sigma \mu^{-1} \*\xi^\it\right)$
    \STATE $\*y^{\it+1} = (1 +\sigma)^{-1}(\*y^\it + \sigma \*\zeta^\it - \sigma\*b)$
    \ENDFOR
    \end{algorithmic}
\end{algorithm}
\imask~is an iterative solver which selects randomly at each iteration a sketched operator $\*K_i$ for the gradient update; in this way the computational cost of the iteration is reduced according to the selected resolution. At each iteration $\it$, an estimated gradient $\*\xi^\it, \*\zeta^\it$ is computed and the proximal operator for both the data fit and the regularizer applied. Here we make implicitly the assumption that the the proximal operator of $R$ \cite{combettes2011proximal, Bauschke2011, parikh2014proximal}, which is defined as
\begin{equation}\label{eq:prox_op}
    \mathrm{prox}_{\sigma R}(\*x') = \arg\min_{\*x\in\R^d} \left\{\frac{1}{2}\|\*x - \*x'\|^2 + \sigma R(\*x)\right\},
\end{equation}
can be computed in reasonable time. We further assume that the regularizer is $\mu$-strongly convex, i.e., $R - \frac{\mu}{2}\|\cdot\|^2$ is convex. 

Associated to the family of operator sketches $\{\*K_i\}$ are two operators which capture its correlation. To this end, we define $\bar{\*K}_1: \R^d \to (\R^m)^r$ and $\bar{\*K}_2: \R^d \to (\R^m)^r$ as 
\begin{align*}
    (\bar{\*K}_1\*x)_i = p_i^{1/2}\*K_i \*x\quad \text{and}\quad (\bar{\*K}_2\*x)_i = p_i\*K_i \*x .
\end{align*}
We are now in a position to state the main convergence theorem for \imask, which is derived from the more general setting in \cref{thm:normalized}.

\begin{theorem} \label{thm:convImaSk}
    Let $\*x^\it$ be defined by \cref{table:ImaSk} and $\*x^*$ the solution of \cref{eq:min_obj} with a $\mu$-strongly convex regularizer $R$. Then a constant $C > 0$ exists, such that after $\It$ iterations, for any $a, b > 0$ the iterates satisfy
    \begin{align*}
    \Expect \|\*x^\It-\*x^*\|^2 \leq \theta^\It C
    \end{align*}
    with the convergence factor
    \begin{align*}
    \theta := \max\left(\frac{\mu + \sigma^2(\|\*K\|^2 + (1+a)\|\bar{\*K}_1\|^2)}{\mu (1 + \sigma)^2} + \frac{b}{\mu} \|\bar{\*K}_2\*x\|^2, \frac{\left(1+a^{-1}\right)\sigma^2}{b(1 + \sigma)^2} + 1 - \min_i p_i\right).
    \end{align*}
\end{theorem}

Note that \cref{thm:convImaSk} does not state the range of permissible step-sizes $\sigma$ such \imask~converges. This is discussed in detail in \cref{sec:range}. Furthermore, optimal step-sizes are discussed in \cref{sec:opt:step}. 

In the following \cref{sec:saddle-point}, we generalize problem \cref{eq:min_obj_new}, propose an algorithm how to solve it and prove its linear convergence. The convergence theory of \imask~then follows as a special case.

%% file: Sec3_Convergence_Analysis.tex
\section{Randomized Algorithm for Saddle-Point Problems}\label{sec:saddle-point}

Let $\XX$ and $\YY$ be Hilbert spaces of any dimension and consider 
\begin{equation}
    \min_{\*x\in\XX}\left\{f(\*A\*x) + g(\*x) \right\}, \label{eq:opt:general}
\end{equation}
where $f : \YY \to \RI, g : \XX \to \RI$ are convex functions and $\*A : \XX \to \YY$ is a bounded linear operator. We consider solving the general optimization problem \cref{eq:opt:general} via a family of operators $\*A_i: \XX \to \YY$, $i = 1, \ldots, n$ such that
$\*A = \sum_{i=1}^n \*A_i.$ 
By construction $\*A_i/p_i$ is an unbiased estimator of $\*A$, i.e., $\Expect_{i}\left[\*A_i/p_i \right]=\*A$. Then \cref{eq:opt:general} can be reformulated as
\begin{equation}\label{eq:min_obj_eq}
    \min_{\*x\in\XX}\left\{\Psi(\*x) := f\left(\sum_{i=1}^n \*A_i\*x\right) + g(\*x) \right\} .
\end{equation}
A close look at \cref{eq:min_obj_eq} shows that this objective function $\Psi$ is not affine in the sampling. This prevents the use of traditional stochastic gradient descent (SGD) solvers which usually use "unbiased" samples of gradients.

In order to operate in a regime amendable to current stochastic optimization algorithms, instead of solving \cref{eq:min_obj_eq}, we reformulate the original minimization as a saddle-point problem by using the convex conjugate of $f$ \cite{Bauschke2011}, 
$f^*(\*y) := \sup_{\*z\in\mathrm{dom}f} \Big\{\langle \*z, \*y\rangle - f(\*z)\Big\}$. Using the convex conjugate, it results that solving \cref{eq:min_obj_eq} is equivalent to finding the primal vector $\*x\in\XX$ of a solution to the saddle point problem 
\begin{equation}\label{eq:primal_dual_min}
\min_{\*x\in\XX}\max_{\*y\in\YY} \left\{\sum_{i=1}^n \langle \*y, \*A_i\*x\rangle  - f^*(\*y) + g(\*x) \right\}.
\end{equation}
It is important to remark that the problem \cref{eq:primal_dual_min} cannot be solved by other stochastic primal-dual algorithms, such as the Stochastic Primal-Dual Hybrid Gradient algorithm (SPDHG) \cite{chambolle2004algorithm} since they require the objective to be separable in the dual variable, i.e., $f$ and $\*y$ are indexed by $i$. 

Since the proposed strategy for \cref{eq:primal_dual_min} is built upon SAGA-type gradient estimators, we review the SAGA method for optimization.

\subsection{SAGA for Optimization}
We summarize in this section the SAGA algorithm \cite{defazio2014} for minimizing the optimization problem
\begin{align}\label{eq:opt:saga}
	\min_{\*x \in \XX} \left\{\sum_{i=1}^n f_i(\*x) + g(\*x) \right\}
\end{align}
with $f(\*x) = \sum_{i=1}^n f_i(\*x)$, each $f_i$ being convex with Lipschitz continuous gradient and $g$ is convex but potentially non-differentiable and we assume that the proximal operator of $g$ exists as defined  in \cref{eq:prox_op}. The SAGA algorithm is reported in \cref{table:SAGA_primal}. 
\begin{algorithm}[H]
    \caption{SAGA for optimization \cref{eq:opt:saga}} \label{table:SAGA_primal}
    \begin{algorithmic}[1]
	\REQUIRE number of iterations $\It \in \N$, probabilities $p = (p_i)_{i=1}^n \in [0,1]^n$, step size $\sigma > 0$
	\INPUT initial iterate $\*x^0 \in \XX$, gradient memory $\*\phi^0 = (\nabla f_i(x^0)_{i=1}^n\in \XX^n$ 
	\OUTPUT Reconstructed signal $\*x^\It$ after $\It$ iterations
	\FOR{$\it = 0, \ldots, \It-1$}
	\STATE Sample $i_\it \in \{1, \dots, n\}$ at random with probability $p$.
	\STATE $\phi^{\it+1}_i = \begin{cases} \nabla f_{i_\it}(\*x^\it) & \text{if $i = i_\it$} \\ \phi^\it_i & \text{else} \end{cases}$
	\STATE $\*x^{\it+1} = \mathrm{prox}_{\sigma g}\,\left(\*x^\it - \sigma \left(\frac{1}{p_{i_\it}}(\phi^{\it+1}_{i_\it} - \phi^\it_{i_\it}) + \sum_{i=1}^n \phi^\it_i\right)\right)$
	\ENDFOR
    \end{algorithmic}
\end{algorithm}

SAGA is initialized with some known vector $\*x^0\in \XX$ and known derivatives  $\*\phi^0_i\in \XX^n$ for each $i\in [n]$ and probabilities $p = (p_i)_{i=1}^n \in [0,1]^n$. The derivatives can be stored in a table data-structure of length $n$. At each iteration $\it$ a sample $i_\it$ from the discrete distribution $p$ is selected and the entry $\*\phi^{\it+1}_i$ is updated in the stored table. In line 5 of \cref{table:SAGA_primal}, the vector $\*x$ is updated in the negative direction of the estimated gradient with step size $\sigma$. 

By observing \cref{eq:min_obj_eq}, the primal problem, induced by the image domain sketch, considered in this work is not separable as in \cref{eq:opt:saga} therefore the SAGA algorithm cannot be utilized in the primal domain. 
In the following section we present the SAGA algorithm for solving the saddle point problem \cref{eq:primal_dual_min} which exploit a primal-dual optimization procedure.

\subsection{SAGA for Saddle-Point Problems}
Now consider again the saddle point-minimization problem \cref{eq:primal_dual_min}. 
We extend SAGA to a setting where in each iteration we update a sketched version of the primal and dual variable. The random variable $i$ taking values in $\{1, \dots, n\}$ with probability $p_i$ is associated to $\*A_i$. The idea is that at each iteration of the algorithm, a sample of $i$ with uniform or non-uniform probability $p_i$ is generated and the linear operator $\*A_i$ is selected. The algorithm we propose is formalized in \cref{table:SAGA_mres}. 

\begin{algorithm}[H]
    \caption{SAGA for saddle-point optimization \cref{eq:primal_dual_min}} \label{table:SAGA_mres}
\begin{algorithmic}[1]
    \REQUIRE number of iterations $\It$, probabilities $p = (p_i)_{i=1}^n \in [0,1]^n$, step sizes $\sigma_x, \sigma_y$
    \INPUT initial iterate $\*x^0 \in \XX$, $\*y^0 \in \YY$, memory $\*\phi^{0} = (\*A_i^*\*y^0)_{i=1}^n \in \XX^n$, $\*\psi^{0} = (\*A_i\*x^0)_{i=1}^n \in \YY^n$,   
    \OUTPUT $(\*x^\It, \*y^\It)$
    \FOR{$\it = 0, \ldots, \It-1$}
    \STATE Sample $i_\it \in \{1, \dots, n\}$ at random with probability $p$
    \STATE $(\*\phi^{\it+1}_i, \*\psi^{\it+1}_i) = \begin{cases} (\*A^*_{i_\it} \*y^\it, \*A_{i_\it} \*x^\it) & \text{if $i = i_\it$} \\ (\*\phi^\it_i, \*\psi^\it_i) & \text{else} \end{cases}$
    \STATE $\*x^{\it+1} = \mathrm{prox}_{\sigma_x g}\,\left(\*x^\it - \sigma_x \left( \frac{1}{p_{i_\it}}\left(\*\psi^{\it+1}_{i_\it} - \*\psi^\it_{i_\it}\right) + \sum_{i=1}^n   \*\psi^\it_i\right)\right)$
    \STATE $\*y^{\it+1} = \mathrm{prox}_{\sigma_y f^*}\,\left(\*y^\it + \sigma_y \left(\frac{1}{p_{i_\it}}\left(\*\phi^{\it+1}_{i_\it} - \*\phi^\it_{i_\it}\right) + \sum_{i=1}^n \*\phi^\it_i\right)\right)$
    \ENDFOR
\end{algorithmic}
\end{algorithm}
It is worth mentioning the following observation to reduce the memory requirement.

\begin{remark}
Instead of $\*A_i \*x^\it$ one could also store just $\*x^\it$ depending on what is preferred memory-wise. This can be implemented with the same number of matrix-vector products due to the linearity of $\*A_i$. Of course, a similar argument can be applied regarding $\*A_i^* \*y^\it$ and $\*y^\it$. 
\end{remark}

\subsection{Normalization}\label{A:scaling}
We consider the case when the functions $f^*$ and $g$ are strongly convex but $\mu_{f^*}$ and $\mu_g$ are different. Following \cite{Chambolle2011pdhg, Balamurugan2016} we normalize the problem such that the rescaled versions $\tilde f^*$ and $\tilde g$ have uniform strong-convexity constants equal to one.
\renewcommand{\t}[1]{\*{\tilde #1}}
Let
\begin{align}\label{eq:var_rescale}
    \t A := (\mu_g\mu_{f^*})^{-1/2}\*A,\quad
	{\tilde f}^*(\*y) := f^*\left(\mu_{f^*}^{-1/2} \*y \right), \quad \text{and} \quad \tilde g(\*x) := g \left(\mu_g^{-1/2} \*x \right).
\end{align}
Both $\tilde f^*$ and $\tilde g$ are $1$-strongly convex and $\|\t A\|^2 = \|\*A\|^2/(\mu_g \mu_{f^*})$ is the condition number of the original problem. Moreover, since
\begin{align*}
    \langle\*A \*x, \*y \rangle &= \langle\t A \mu_g^{1/2} \*x, \mu_{f^*}^{1/2} \*y \rangle, \quad
    f^*(\*y) = \tilde f^*\left(\mu_{f^*}^{1/2} \*y\right), \quad \text{and} \quad 
    g(\*x) = \tilde g\left(\mu_g^{1/2} \*x\right),
\end{align*}
problem \cref{eq:primal_dual_min} is equivalent to
\begin{equation}
    \min_{\t x\in\XX}\max_{\t y\in \YY} \left\{\sum_{i=1}^n \langle\t y, \t A \t x\rangle - \tilde f^*\left(\t y\right) + \tilde g\left(\t x\right).\right\} \label{eq:new}
\end{equation}
A solution to \cref{eq:primal_dual_min} can be recovered via $\*x^* = \mu_g^{-1/2} \t x^*$ and $\*y^* = \mu_{f^*}^{-1/2} \t y^*$, where $(\t x^*, \t y^*)$ solves~\cref{eq:new}.

Here we use the rescaling only for the theoretical analysis, but the proximal operator of the scaled variant is as easy to compute as of the original version, e.g.,
\begin{equation*}
    \mathrm{prox}_{\sigma \tilde f^*}(\*p)
	= \mu_{f^*}^{1/2} \mathrm{prox}_{\sigma \mu_{f^*}^{-1} f^*} \left(\mu_{f^*}^{-1/2}\*p \right). 
\end{equation*}
Then the normalized SAGA for saddle-point optimization is given in \cref{alg:NImaSk}. 

\begin{algorithm}
    \caption{Normalized SAGA for saddle-point optimization \cref{eq:new}.} \label{alg:NImaSk}
    \begin{algorithmic}[1]
    \REQUIRE number of iterations $\It$, probabilities $p = (p_i)_{i=1}^n \in [0,1]^n$
    \INPUT initial iterate $\t x^0 \in \XX$, $\t y^0 (= \* 0)$, memory $\t \phi^{0} = (\t A_i\t x^0)_{i=1}^n \in \YY^n$, $\t \psi^{0} = (\t A_i^*\t y^0)_{i=1}^n \in \XX^n$  
    \OUTPUT $(\t x^\It, \t y^\It)$
    \FOR{$\it = 0, \ldots, \It-1$}
    \STATE Sample $i_\it\in \{1, \dots, n\}$ at random with probability $p$
    \STATE $(\t\phi^{\it+1}_i, \t\psi^{\it+1}_i) = \begin{cases} (\t A_{i_\it} \t x^\it, \t A^*_{i_\it} \t y^\it) & \text{if $i = i_t$} \\ (\t\phi^t_i, \t\psi^t_i) & \text{else} \end{cases}$
    \STATE $\t x^{\it+1} = \mathrm{prox}_{\sigma \tilde g}\,\left(\t x^\it - \sigma \left(\frac{1}{p_{i_\it}}\left(\t\psi^{\it+1}_{i_\it} - \t\psi^\it_{i_\it}\right) + \sum_{i=1}^n \t\psi^\it_i\right)\right)$
    \STATE $\t y^{\it+1} = \mathrm{prox}_{\sigma \tilde f^*}\,\left(\t y^\it + \sigma \left(\frac{1}{p_{i_\it}}\left(\t\phi^{t+1}_{i_\it} - \t\phi^t_{i_\it}\right) + \sum_{i=1}^n \t\phi^\it_i\right)\right)$
    \ENDFOR
    \end{algorithmic}
\end{algorithm}

\subsection{Convergence Analysis}\label{ch:convergence}

It is convenient to analyze the convergence via a single-variable formulation, see \cref{alg:NImaSk:z}. We define the joint variable $\*z = (\t x, \t y) \in \ZZ := \XX \times \YY$ where the $\ZZ$ is equipped with the inner product $\langle (\t x, \t y), (\t x', \t y')\rangle = \langle \t x, \t x' \rangle  + \langle \t y, \t y'\rangle$. Then we define 
\begin{align}
    h(\*z) := \tilde g(\t x) + \tilde f^*(\t y) \quad \text{and} \quad \*B := \left[\begin{array}{cc}
    \*0  & \t A^* \\ 
    -\t A & \*0
    \end{array}\right] \label{eq:z:hB}.
\end{align}
We denote its norm by $L := \|\*B\|$. Note that $\* B = \sum_{i=1}^n \* B_i$ with 
\begin{align*}
	\* B_i := \left[\begin{array}{cc}
		\*0  & \t A^*_i \\ 
		-\t A_i & \*0
	\end{array}\right].
\end{align*}

\begin{algorithm}[!ht]
    \caption{Normalized SAGA for saddle-point optimization \cref{eq:new} in joint variable $\*z$.}
	\label{alg:NImaSk:z}
	\begin{algorithmic}[1]
		\REQUIRE number of iterations $\It$, probabilities $p = (p_i)_{i=1}^n \in [0,1]^n$, step size $\sigma > 0$
		\INPUT initial iterate $\*z^0 \in \ZZ$, memory $\* g^{0} = (\* B_i\* z^0)_{i=1}^n \in \ZZ^n$
		\OUTPUT $\* z^\It$
		\FOR{$\it = 0, \ldots, \It-1$}
		\STATE Sample $i_\it\in \{1, \dots, n\}$ at random with probability $p$.
		\STATE $\* g^{\it+1}_i = \begin{cases} \* B_{i_\it} \* z^\it & \text{if $i = i_\it$} \\ \* g^\it_i & \text{else} \end{cases}$
		\STATE $\* z^{\it+1} = \mathrm{prox}_{\sigma h}\,\left(\* z^\it - \sigma \left(\frac{1}{p_{i_\it}}(\*g^{\it+1}_{i_\it} - \*g^\it_{i_\it}) + \sum_{i=1}^n \*g^\it_i\right)\right)$
		\ENDFOR
	\end{algorithmic}
\end{algorithm}

Associated to $\*B$ is two further operators related to the correlations between the $\*B_i$.
Let $\bar{\*B} : \ZZ \to \ZZ^n, (\bar{\*B}\*z)_i := \*B_i \*z$ and its size $\bar{L} := \|\bar{\*B}\|$ will be an important constant in the convergence analysis. Note that $\bar{L}$ can alternatively be defined as
\begin{equation}\label{eq:def_Lbar}
	\bar{L}^2 = \sup_{\|\*z\| \leq 1} \sum_{i=1}^n \|\*B_i\*z\|^2 .
\end{equation}
Furthermore, another important constant will be $\bar L_p$ defined as the norm of $\bar{\*B}_p : \ZZ \to \ZZ^n, (\bar{\*B}_p\*z)_i := p_i^{-1/2} \*B_i \*z$.
The operator $\*B$ has some properties that will be exploited in the convergence analysis.
\begin{lemma} \label{lem:operatorB}
    Let $\*B$ be defined as in \cref{eq:z:hB}. Then $L = \|\t A\|$ and $\*B$ is skew-symmetric, in particular $\langle \*z, \*B\*z \rangle = 0$.
    
    \begin{proof}
    Regarding the first assertion, given $\*z = (\*x, \*y)$, we have
    \begin{align*}
    \|\*B\*z\|^2
    &= \left\| (\t A^*\*y, -\t A\*x) \right\|^2
    = \|\t A^*\*y\|^2 + \|\t A\*x\|^2
    \leq \|\t A\|^2 (\|\*y\|^2 + \|\*x\|^2) = \|\t A\|^2 \|\*z\|^2,
    \end{align*}
    thus, $\|\*B\| \leq \|\t A\|$. Now, for $\*z = (\*x, \*y)$ with $\|\t A\*x\|= \|\t A\| \|\*x\|$ and $\*y = 0$, we have
    \begin{align*}
    \|\*B\*z\| = \|\t A\*x\| = \|\t A\|\|\*x\|= \|\t A\|\|\*z\|.
    \end{align*}
    Therefore, the upped bound is obtained, leading to $L = \|\*B\| = \|\t A\|$.
		
    Regarding the second assertion,
    \begin{align*}
    \langle \*z, \*B\*z \rangle 
    = \langle (\*x, \* y), (\t A^* \*y, -\t A \*x) \rangle
    = \langle \*x, \t A^* \* y \rangle - \langle \*y, \t A \*x \rangle
    = 0.
    \end{align*}
    \end{proof}
\end{lemma}

\begin{remark}
In order to understand these constants better, consider the case of uniform sampling $p_i = 1/n$ and the same operator $\t A_i = \*C, i = 1, \dots, n$. Then $L = \bar L_p = n \|\*C\|$ and $\bar L = \sqrt{n} \|\*C\|$.
\end{remark}

We are now in a position to state the main convergence theorem for \cref{alg:NImaSk:z}.

\begin{theorem} \label{thm:normalized}
    Let $\* z^k$ be the iterations defined by \cref{alg:NImaSk:z} and $\*z^*$ the saddle point of \cref{eq:new}. Then after $K$ iterations, for any $a, b > 0$ the iterates satisfy
	\begin{align*}
		\Expect \|\*z^K-\*z^*\|^2 \leq \theta^K C
	\end{align*}
	with the convergence factor
	\begin{align*}
		\theta := \max\left(\frac{1 + \sigma^2(L^2 + (1+a)\bar{L}_p^2)}{(1 + \sigma)^2} +  b \bar L^2, \frac{\left(1+a^{-1}\right)\sigma^2}{b(1 + \sigma)^2} + 1 - \min_i p_i\right)
	\end{align*}
	and constant $C := \|\*z^0-\*z^*\|^2 + b \sum_{i=1}^n \frac{1}{p_i} \|\*B_i\*z^* - \*g^0_i\|^2$.
\end{theorem}

We prove \cref{thm:normalized} by combining \cref{lem:step1} and \cref{lem:step2}. In order to simplify the analysis in \cref{lem:step1} we use the following classical lemma which estimates the variance by its second-order moment.
\begin{lemma}\label{lem:var}
	Let be $X$ a random variable, then $\Expect\|X - \Expect X\|^2\leq \Expect\|X\|^2$.
\end{lemma}

\begin{lemma} \label{lem:step1}
    Let $\Expect_{i_\it}$ denote the expectation with respect to $i_\it$ and define the error in the iterates by $\Delta_\it := \|\*z^\it - \*z^*\|^2$ and the error in the gradient by $\Gamma_\it := \sum_{i=1}^n \frac{1}{p_i} \|\*B_i\*z^* - \*g^\it_i\|^2$ where $\*z^\it$ are the iterates defined via \cref{alg:NImaSk:z}.
	
    Then for all $\It \in \mathbb N$ and $a > 0$ it holds that
    \begin{align*}
    \Expect_{i_\it}\Delta_{\it+1} \leq
    \frac{1 + \sigma^2(L^2 + (1+a)\bar{L}_p^2)}{(1 + \sigma)^2} \Delta_\it + \frac{ (1+a^{-1})\sigma^2}{(1 + \sigma)^2} \Gamma_\it .
    \end{align*}
    \begin{proof}
    Note that the saddle point $\*z^*$ of \cref{eq:new} satisfies $\*z^* = \mathrm{prox}_{\sigma h}\big(\*z^* - \sigma\*B\*z^*\big)$. 
    Let $\*f_{i_\it}^\it := \frac{1}{p_{i_\it}}\left(\*B_{i_\it}\*z^\it - \*g_{i_\it}^\it\right)$ with expected value $\Expect_{i_\it} \*f_{i_\it}^\it = \*B\*z^\it - \sum_{i=1}^n\*g^\it_i$. Then we can write $\*z^{\it+1} = \operatorname{prox}_{\sigma h}(\*z^\it - \sigma ( \*f_{i_\it}^\it + \sum_{i=1}^n\*g^\it_i))$. Using the fact that the proximal mapping of strongly convex functions with constant 1 is contractive \cite{parikh2014proximal} in the sense that
    \begin{align*}
    \|\mathrm{prox}_{\sigma h}(\*z) - \mathrm{prox}_{\sigma h}(\*z')\| \leq (1 + \sigma)^{-1}\| \*z - \* z'\|,
    \end{align*}
    we derive the first estimate on the error $\*e^\it := \*z^\it - \*z^*$ as
    \begin{align*}
    \|\*e^{\it+1}\| 
    &\leq \frac{1}{1 + \sigma}\left\|\*z^\it - \sigma \left(\*f_{i_\it}^\it + \sum_{i=1}^n\*g^\it_i\right) - \left(\*z^* - \sigma \*B\*z^*\right)\right\| \\
    &= \frac{1}{1 + \sigma}\left\|\*e^\it - \sigma\*B\*e^\it - \sigma\left( \*f_{i_\it}^\it - \Expect_{j_\it} \*f_{j_\it}^\it\right)\right\|.
    \end{align*}
    Using elementary properties of the expectation yields
    \begin{align}
    \Expect_{i_\it} \Delta_{\it+1} 
    &\leq \frac{1}{(1 + \sigma)^2}\Expect_{i_\it}\left\|\*e^{\it} - \sigma\*B \*e^{\it} + \sigma\left(\*f_{i_\it}^\it - \Expect_{j_\it} \*f_{j_\it}^\it\right) \right\|^2 \nonumber \\
    &= \frac{1}{(1 + \sigma)^2}\|\*e^{\it} - \sigma\*B \*e^{\it}\|^2 + \frac{\sigma^2}{(1 + \sigma)^2} \Expect_{i_\it}\|\*f_{i_\it}^\it - \Expect_{j_\it} \*f_{j_\it}^\it\|^2. \label{eq:exp_sq_norm}
    \end{align}
    We analyze the deterministic and random terms in \cref{eq:exp_sq_norm} separately. Using \cref{lem:operatorB} we can bound the deterministic term as
    \begin{equation}\label{eq:det_term}
	\|\*e^{\it} - \sigma\*B \*e^{\it}\|^2
	= \|\*e^{\it}\|^2 - 2\sigma\langle \*e^{\it}, \*B\*e^{\it}\rangle + \sigma^2\|\*B\*e^{\it}\|^2 \leq (1 + \sigma^2L^2) \Delta_\it.
    \end{equation}
    Using \cref{lem:var} and the definition of the expectation we get
    \begin{align*}
	\Expect_{i_\it}\|\*f_{i_\it}^\it - \Expect_{j_\it} \*f_{j_\it}^\it\|^2
    \leq \Expect_{i_\it} \|\*f_{i_\it}^\it\|^2
    = \sum_{i=1}^n \frac{1}{p_i} \|\*B_i\*z^{\it} - \*g^{\it}_i\|^2.
    \end{align*}
    With $\|\*a + \*b \|^2 \leq (1+a)\|\*a\|^2 + (1+a^{-1})\|\*b\|^2$ for any $a > 0$ we can further estimate
    \begin{align}\label{eqn:rand}
    \Expect_{j_\it}\|\*f_{j_\it}^\it - \Expect_{i_\it} \*f_{i_\it}^\it\|^2 \nonumber
    &\leq \sum_{i=1}^n \frac{1}{p_i} \|\*B_i\*z^\it - \*g^\it_i\|^2
    = \sum_{i=1}^n \frac{1}{p_i} \|\*B_i\*e^\it + \*B_i\*z^* - \*g^\it_i\|^2 \\
    &\leq \sum_{i=1}^n \frac{1 + a}{p_i} \|\*B_i\*e^\it\|^2 + \sum_{i=1}^n \frac{1 + a^{-1}}{p_i} \|\*B_i\*z^* - \*g^\it_i\|^2 \nonumber \\
    &\leq (1+a)\bar{L}_p^2 \Delta_\it + (1+a^{-1}) \Gamma_\it.
    \end{align}
    Inserting \cref{eq:det_term} and \cref{eqn:rand} into \cref{eq:exp_sq_norm} yields the assertion.
    \end{proof}
\end{lemma}

\begin{lemma} \label{lem:step2}
    Using the same notation as in \cref{lem:step1} it holds that
    \begin{align*}
    \Expect_{i_\it} \Gamma_{\it+1} \leq \bar L^2 \Delta_\it + (1 - \min_i p_i) \Gamma_\it.
    \end{align*}
    \begin{proof}
    To analyze $\Gamma_\it$, we note that $\*g^{\it+1}_i =\*g^\it_i$ if $i\neq i_\it$, while $\*g^{\it+1}_{i_\it} = \*B_{i_\it}\*z^\it$. Thus,
    \begin{align*}
    \Expect_{i_\it} \Gamma_{\it+1}
    &= \Expect_{i_\it}\left(\sum_{i\neq i_\it} \frac{1}{p_i} \|\*B_i\*z^* -\*g^\it_i\|^2 + \frac{1}{p_{i_\it}} \|\*B_{i_\it}\*z^* - \*B_{i_\it}\*z^\it\|^2\right) \\
    &= \sum_{i=1}^n \frac{1}{p_i} \|\*B_i\*z^* -\*g^\it_i\|^2 + \Expect_{i_\it}\left( \frac{1}{p_{i_\it}} \|\*B_{i_\it}\*z^* - \*B_{i_\it}\*z^\it\|^2 - \frac{1}{p_{i_\it}}\|\*B_{i_\it}\*z^* -\*g^\it_{i_\it}\|^2\right) \\
    &= \Gamma_\it + \sum_{i=1}^n \|\*B_i\*e^\it\|^2 - \sum_{i=1}^n \|\*B_i\*z^* -\*g^\it_i\|^2
    \end{align*}
    and the assertion follows.
    \end{proof}
\end{lemma}

\begin{proof}[Proof of \cref{thm:normalized}]
    Taking a linear combination of the estimates in \cref{lem:step1} and \cref{lem:step2} yields
    \begin{align*}
    \Expect_{i_\it} (\Delta_{\it+1} + b \Gamma_{\it+1})
    &\leq \left(\frac{1 + \sigma^2(L^2 + (1+a)\bar{L}_p^2)}{(1 + \sigma)^2} +  b \bar L^2\right) \Delta_\it + \left(\frac{\left(1+a^{-1}\right)\sigma^2}{b(1 + \sigma)^2} + 1 - \min_i p_i\right) b \Gamma_\it \\
    &\leq \theta (\Delta_\it + b \Gamma_\it).
    \end{align*}
    Taking expectations, we recursively derive
    \begin{align*}
    \Expect (\Delta_\It + b \Gamma_\It) \leq \theta^\It (\Delta_0 + b \Gamma_0).
    \end{align*}
    The assertion follows since $\Gamma_\It$ is nonnegative.
\end{proof}

\subsection{Range of Admissible Step-Sizes} \label{sec:range} 
In the following we will discuss the consequences of \cref{thm:normalized} on the step-size parameter $\sigma$.
It is convenient to reparameterize the step-size as $\eta = \sigma/(1 + \sigma) \in (0, 1)$ such that $1 - \eta = 1/(1 + \sigma)$
and the convergence factor becomes
\begin{align}\label{eq:theta}
	\theta_{a,b}(\eta) = \max(\phi_{a,b}(\eta), \psi_{a,b}(\eta))
\end{align}
with
\begin{align}\label{eq:phi_psi}
	\phi_{a,b}(\eta) &:= c^{-1} (\eta - c)^2 + 1 + b \bar L^2 - c, \quad \psi_{a,b}(\eta) := \frac{1+a^{-1}}{b} \eta^2 + 1 - \min_i p_i
\end{align}
and $c = (1 + L^2 + (1+a)\bar{L}_p^2)^{-1}$. Note that since $a > 0$ we have $c \in (0, \bar c)$ with $\bar c := (1 + L^2 + \bar{L}_p^2)^{-1} \leq 1$. 
Since $\phi_{a,b}$ is minimized at $c$ with $\phi_{a,b}(c) = 1 + b \bar L^2 - c$, we have immediately a necessary condition on $b$. 
\begin{lemma} \label{lem:nec:b}
    If $b \bar L^2 \geq c$, then for all $\eta \geq 0$ it holds that $\phi_{a,b}(\eta) \geq 1$. Similarly, if $b \bar L^2 < c$, there exists $\eta \geq 0$ such that $\phi_{a,b}(\eta) < 1$.
\end{lemma}

Due to \cref{lem:nec:b} and the aforementioned discussion, it makes sense to reparametrize $\phi_{a,b}$ and $\psi_{a,b}$. To this end, let $\rho \in (0, 1)$ and $b = \rho c / \bar L^2$. Then we have $b \bar L^2 = \rho c < c$ and the assertion of \cref{lem:nec:b} is naturally satisfied. Using $b = \rho c / \bar L^2$, we arrive at 
\begin{align*}
	\phi_{\rho, c}(\eta) &= c^{-1} (\eta - c)^2 + 1 - (1 - \rho) c.
\end{align*}
Of course, here we abused the notation by reusing the symbol $\phi$ but it should be clear from the context which parametrization is being used.

Similarly, we can express $a^{-1}$ in terms of $c$ as
\begin{align*}
	a^{-1} = \frac{\bar{L}_p^2c}{1 - (1 + L^2 + \bar{L}_p^2)c}
\end{align*}
and thus, with
\begin{align}
    \alpha^{-1} := \frac{1+a^{-1}}{b} 
    = \frac{1 - (1 + L^2)c}{(1 - (1 + L^2 + \bar{L}_p^2)c)b}
    = \frac{\bar L^2}{\rho c} \frac{1 - (1 + L^2)c}{1 - (1 + L^2 + \bar{L}_p^2)c} \label{eq:alpha}
\end{align}
we have
\begin{align*}
    \psi_{c, \rho}(\eta) &= \alpha^{-1}\eta^2  + 1 - \min_i p_i.
\end{align*}

With these reparametrizations, one can readily see the admissible ranges.
\begin{lemma} \label{lem:range:phi}
    Let $c \in (0, \bar c), \rho \in (0, 1)$. Then $\phi_{c, \rho}(\eta) < 1$ if and only if $$\eta \in (c - c \sqrt{1 - \rho}, c + c \sqrt{1 - \rho}).$$
    \begin{proof}
    This result follows immediately from solving the quadratic equation $\phi_{c, \rho}(\eta) = 1$ in terms of $\eta$.
    \end{proof}
\end{lemma}

\begin{lemma} \label{lem:range:psi}
    Let $c \in (0, \bar c), \rho \in (0, 1)$. Then $\psi_{c, \rho}(\eta) < 1$ if and only if $\eta < \sqrt{\alpha \min_i p_i}$.
    \begin{proof}
    This result is immediately obtained by solving the quadratic equation $\psi_{c, \rho}(\eta) = 1$ in terms of $\eta$.
    \end{proof}
\end{lemma}

\begin{proposition}
    Let $c \in (0, \bar c), \rho \in (0, 1)$. Then $\theta_{c, \rho}(\eta) < 1$ if and only if $$\eta \in (c - c \sqrt{1 - \rho}, c + c \sqrt{1 - \rho}) \cap (0, \sqrt{\alpha} \min_i \sqrt{p_i}).$$
    \begin{proof}
    $\theta_{c, \rho}(\eta) < 1$ is the case if and only if $\phi_{c, \rho}(\eta) < 1$ and $\psi_{c, \rho}(\eta) < 1$. Then the result follows directly from \cref{lem:range:phi} and \cref{lem:range:psi}.
    \end{proof}
\end{proposition}

\subsection{Optimal Step-Sizes} \label{sec:opt:step}
Notice that $\phi_{c, \rho}$ and $\psi_{c, \rho}$ are quadratics with strictly positive quadratic constants. Further since $\phi_{c, \rho}$ is minimized at $c$ with $\phi_{c, \rho}(c) = 1 - (1 - \rho) c$ and $\psi_{c, \rho}$ is minimized at 0 with $\psi_{c, \rho}(0) = 1 - \min_i p_i$, we derive the following characterization of the optimal step-size:
\begin{lemma} \label{lem:opt}
    There is an unique optimal $\eta^* \in [0, c]$ that minimizes the convergence factor~$\theta_{c, \rho}$. Furthermore,
    \begin{align*}
    \theta_{c, \rho}(\eta^*) \geq 1 -  \min\{(1 - \rho) c, p_1, \dots, p_n\}.
    \end{align*}
\end{lemma}

\begin{theorem} \label{thm:optsteps}
    There is an unique optimal $\eta^* \in (0, c]$ that minimizes the convergence factor~$\theta_{c, \rho}$. Furthermore, if 
\begin{align}
\min_i p_i > \left(\frac{\bar L^2}{\rho} \frac{1 - (1 + L^2)c}{1 - (1 + L^2 + \bar{L}_p^2)c} + 1-\rho\right) c, \label{eq:opt:cond}
\end{align}
    then $\eta^* = c$ and $\theta_{c, \rho}(\eta^*) = 1 - (1-\rho) c$.

    Otherwise, 
    \begin{align}
    \eta^* = \frac{\sqrt{1 + (\alpha^{-1}-c^{-1})(\min_i p_i + \rho c)}-1}{\alpha^{-1}-c^{-1}} \label{eq:eta_opt}
    \end{align}
    and $\theta_{c, \rho}(\eta^*) = \phi_{c, \rho}(\eta^*) = \psi_{c, \rho}(\eta^*)$.
    \begin{proof}

    With \cref{lem:opt} and due to monotonicity, there are only three mutually exclusive potential cases:
\begin{enumerate}
\item $\phi_{c, \rho}(\eta) > \psi_{c, \rho}(\eta)$ for all $\eta \in [0, c]$ iff $\phi_{c, \rho}(c) > \psi_{c, \rho}(c)$ iff $\eta^* = c$ and $\theta^* = \phi_{c, \rho}(c)$
\item $\phi_{c, \rho}(\eta) < \psi_{c, \rho}(\eta)$ for all $\eta \in [0, c]$ iff $\phi_{c, \rho}(0) < \psi_{c, \rho}(0)$ iff $\eta^* = 0$ and $\theta^* = \psi_{c, \rho}(0)$
\item $\eta^* \in (0, c)$ and $\theta^* = \phi_{c, \rho}(\eta^*) = \psi_{c, \rho}(\eta^*)$
\end{enumerate}

The first case is equivalent to \cref{eq:opt:cond}.
Note that this is true for sufficiently small $c$.

The second case is equivalent to $\rho c + \min_i p_i < 0$ which is always false.

In the third case, the optimal step size can be computed as in \eqref{eq:eta_opt}.
    \end{proof}
\end{theorem}

Figure \ref{fig:step-sizecomp} shows the optimal $\sigma$ and its convergence rate $\theta$ as a function of $c$ and $\rho$ as computed by Theorem \ref{thm:optsteps} for the numerical as described in the next section. The constants $L, \bar{L}, \bar{L}_p$ were estimated using the power method \cite{golub2013matrix}
The fastest contraction is given by $\sigma^* \approx 0.0075$ and $\theta^* \approx 0.9925$. Note that this can be compared to 
\begin{equation}
    \sigma_B = \frac{1}{L^2 + 3\bar{L}^2} \approx 0.0127
\end{equation}
as used in \cite{Balamurugan2016}.

\noindent\begin{minipage}{.5\textwidth}
\begin{tikzpicture}[scale=0.7]
\definecolor{darkgray176}{RGB}{176,176,176}
\begin{axis}[
colorbar,
colorbar style={ylabel={}},
colormap/viridis,
point meta max=0.0267855319171118,
point meta min=0.000801164083850264,
tick align=outside,
tick pos=left,
x grid style={darkgray176},
xlabel={\(\displaystyle c\)},
xmin=-0.5, xmax=99.5,
xtick style={color=black},
xtick={0,25,50,74,99},
xticklabels={0.01,0.04,0.06,0.09,0.12},
y dir=reverse,
y grid style={darkgray176},
ylabel={\(\displaystyle \rho\)},
ymin=-0.5, ymax=99.5,
ytick style={color=black},
ytick={0,25,50,74,99},
yticklabels={0.01,0.26,0.50,0.74,0.99}
]
\addplot graphics [includegraphics cmd=\pgfimage,xmin=-0.5, xmax=99.5, ymin=99.5, ymax=-0.5] {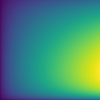};
\end{axis}
\end{tikzpicture}
\end{minipage}
\hspace{-.07cm}
\begin{minipage}{.5\textwidth}
\begin{tikzpicture}[scale=0.7]
\definecolor{darkgray176}{RGB}{176,176,176}
\begin{axis}[
colorbar,
colorbar style={ylabel={}, /pgf/number format/precision=3},
colormap/viridis,
point meta max=1,
point meta min=0.992501666408587,
tick align=outside,
tick pos=left,
x grid style={darkgray176},
xlabel={\(\displaystyle c\)},
xmin=-0.5, xmax=99.5,
xtick style={color=black},
xtick={0,25,50,74,99},
xticklabels={0.01,0.04,0.06,0.09,0.12},
y dir=reverse,
y grid style={darkgray176},
ylabel={\(\displaystyle \rho\)},
ymin=-0.5, ymax=99.5,
ytick style={color=black},
ytick={0,25,50,74,99},
yticklabels={0.01,0.26,0.50,0.74,0.99}
]
\addplot graphics [includegraphics cmd=\pgfimage,xmin=-0.5, xmax=99.5, ymin=99.5, ymax=-0.5] {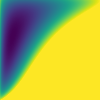};
\end{axis}
\end{tikzpicture}
\end{minipage}
\captionof{figure}{Optimal step-size $\sigma$ (left) and its corresponding convergence rate $\theta$ (right).}\label{fig:step-sizecomp}

%% file: Sec4_Simulations_Results.tex
\newpage 

\section{Numerical Results}\label{sec:results}

\subsection{Efficient Image Sketching} \label{sec:sketches} 
For simplicity, we outline the algorithm in 2D for square images but it can be readily generalized to other dimensions and geometries. Let $\*x \in \R^d$, with $d = 2^n, \, n \in \-Z_+$, be the vectorized version of an $\sqrt{d}\times\sqrt{d}$ image at \enquote{full} resolution which is mapped to measurements $\*b\in\R^m$ by a linear forward operator $\*K\in\R^{m\times d}$. For instance, in CT the operator $\*K$ models the X-ray line tracing geometry from source to detectors, crossing through the discretized image pixels at different angles and maps into projection measurements (sinogram). The measurements dimension $m$ is dictated by the geometry and can be calculated as the product of the size of the detector array and the number of angles. In a naive sequential implementation, this operator scales linearly in the image dimension $\sqrt{d}$. \Cref{fig:ASTRA_comp_time} shows a numerical demonstration, with an implementation of the operator $\*K$ for the X-ray CT problem, that in practice the relationship can be fit as a linear scaling. Thus, in order to reduce the computation, we aim to construct operators $\*K_i\in\R^{m\times d}$ which map an image in a low-resolution discretized grid of dimension $\sqrt{d_i}\times \sqrt{d_i}$ to measurements of the original dimension $m$ which are similar to $\*K$. As long as reducing the dimension is cheap, the computational cost of $\*K_i \*x$ compared to $\*K \*x$ is then about $\sqrt{d}/\sqrt{d_i}$ cheaper. 
We define the operator $\*M$ to convert a  $d$-dimensional vector $\*x \in \R^d$ into a $\sqrt{d}\times\sqrt{d}$ matrix $\*M\*x$. 

The linear mappings from full to low resolution grids in the image domain can be modeled with the decimator operators $\*T_i: \R^d\rightarrow\R^{d_i}, \, i = 1, \ldots, r-1$ from the full-resolution $r$, images $\sqrt{d}\times\sqrt{d}$, to low-resolution $\*M\left[\*T_i\*x\right]$ images by performing downsampling of the factor $2^i\times 2^i$. The simplest of such operators are $2^{i}\times 2^{i}$-block averaging defined as following: 
\begin{equation}	\*M\left[\*T_i\*x\right]\left(j,k\right) = \frac 1 {2^{2i}} \sum_{s=1}^{2^i}\sum_{t=1}^{2^i} \*M\*x\left(2^i(j-1)+s, 2^i(k-1)+t\right), \quad j,k \in \left\{ 1, \ldots , \sqrt{d}/2^i \right\}
\end{equation}
with adjoint
\begin{align*}	\*M\left[\*T^*_i\*x\right]\left(j, k\right) &= \*M\*x\left(\left\lceil j /2^i \right\rceil, \left\lceil k/2^i \right\rceil\right), \quad j,k \in \{1, \ldots, \sqrt{d}\}.
\end{align*}
It is important to note that $\*S_i = \*T_i^T \*T_i$ is proportional to the projection onto the subspace of piece-wise constant images, see also \cref{fig:CTsketch_ASTRA_Idecim}.

Given a full resolution image $\*x\in\R^d$ and the bounded linear operator $\*K$, it is possible to define a family of low resolution linear sketches $\*K_i: \R^d\rightarrow\R^{m}, i = 1, \ldots, r-1$ as following   
\begin{equation}\label{eq:CT_sketches}
	\*K_i\*x = \left\{
	\begin{array}{lcl}
		\*K\*S_i \*x = \*K \*T^T_i \*T_i \*x = \*R_i\*T_i \*x & , & i = 1,\ldots, r-1 \\
		\*K \*S_r \*x & , & i = r    
	\end{array}\right.
\end{equation}
where $p_i$ are the probability of selecting low resolution sketches while $p_r = 1 - \sum_{i=1}^{r-1}p_i$ is the probability of selecting the full resolution operator. 
It is important to highlight that in practice in the low resolution case, $i=1, \ldots, r-1$, the operation $\*K \*T^*_i$ can be approximated with $\*R_i$ which represents the application of the forward model on the $i$th low resolution grid. This means that it is never required to compute matrix-vector multiplications at full resolution. Moreover, the up-sampling $\*T^*_i$ is never explicitly computed. Recalling the CT case, $\*R_i\in\R^{d_i\times m}$ represents the X-ray tracing model from the low dimensional $\sqrt{d_i}\times \sqrt{d_i}$ input grid to the $m$ measurements. 
In \cref{eq:CT_sketches} for the full resolution case, $i=r$, the linear mapping  $\*T^*_r \*T_r$ is derived as a result of the  constraint \cref{eq:si} as
\begin{equation}		
	\*S_r = p_r^{-1} \left(\*I - \sum_{i=1}^{r-1} p_i \*T_i^T\*T_i\right).
\end{equation}

We can visualize and analyze the effect of the described the sketched low-resolution operators with the X-ray CT example. We consider the CT forward operator $\*K$ and the construction of the low resolution sketches $\*K_i$ implemented with the ASTRA Toolbox \cite{vanAarle:16} using a parallel beam geometry with line discretization. We simulate the CT acquisition using $n_{\theta}=100$ angular projections and $n_d=\lceil\sqrt{2d}\rceil$ detectors, where the dimension of the measurement vector $\*b$ is $m = n_{\theta}\cdot n_d$. We further added Gaussian noise to the projection data with a uniform initial count of $10^5$ X-ray photons. The input data consists of a 1-mm pixel-width $512\times 512$ torso axial slice images generated from the XCAT phantom \cite{segars20104d}. 
\cref{fig:CTsketch_ASTRA_Idecim} shows the effect obtained by concatenating the decimation operator $\*T_i$ and interpolator $\*T_i^*$ for $r=4$ to the XCAT ground truth image $\*x^*$ of dimension $\sqrt{d} \times\sqrt{d}$ with $\sqrt{d} = 512$. 

\begin{figure*}[!ht]
	\centering
	\small\addtolength{\tabcolsep}{-8pt}
	\renewcommand{\arraystretch}{0.1}
	
	\begin{tabular}{p{2.8cm} p{2.8cm} p{2.8cm} p{2.8cm} p{2.8cm} p{2.8cm}}
		
		\multicolumn{1}{c}{$\*K\*x^*$}
		&
		\multicolumn{3}{c}{$\*K_i\*x^* = \*R_i\*T_i\*x^*$} 
		&
		\multicolumn{1}{c}{\hspace*{-2mm}$\*K\*S_r\*x^*$}
		& 
		\multicolumn{1}{c}{\hspace*{-5mm}$\*K\*x^*-\sum_{i=1}^{r}p_i\*K_i\*x^*$} \vspace{.1cm} \\		
		\multicolumn{1}{c}{} & \multicolumn{3}{c}{\downbracefill} & \multicolumn{1}{c}{} \vspace{.1cm} & \multicolumn{1}{c}{} \vspace{.1cm} 
		\\
		\includegraphics[clip, trim=10pt 10pt 10pt 10pt, width=2.5cm, height=6cm]{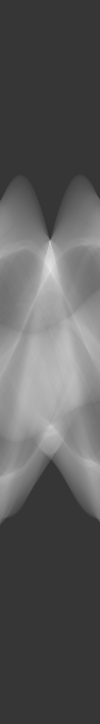}
		&
		\includegraphics[clip, trim=10pt 10pt 10pt 10pt, width=2.5cm, height=6cm]{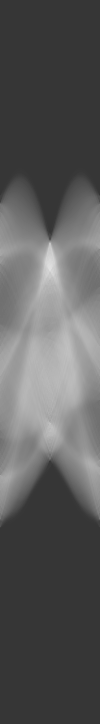}
		&       
		\includegraphics[clip, trim=10pt 10pt 10pt 10pt, width=2.5cm, height=6cm]{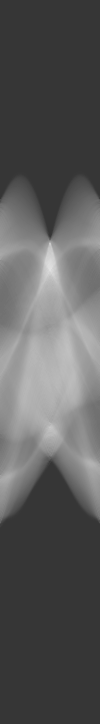} 
		&                          
		\includegraphics[clip, trim=10pt 10pt 10pt 10pt, width=2.5cm, height=6cm]{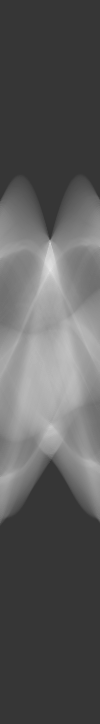} 
		&
		\includegraphics[clip, trim=10pt 10pt 10pt 10pt, width=2.5cm, height=6cm]{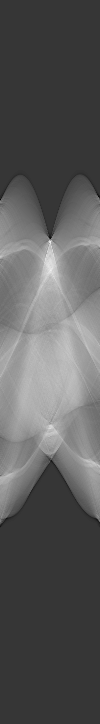} 
		&                             
		\includegraphics[clip, trim=10pt 10pt 10pt 10pt, width=2.5cm, height=6cm]{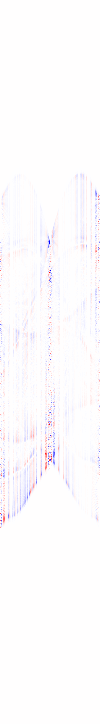}
		\\ [2pt]
		\multicolumn{1}{c}{(a) Full sinogram} 
		& 
		\multicolumn{1}{c}{(b) $\sqrt{d_1} = 64$} 
		& 
		\multicolumn{1}{c}{(c) $\sqrt{d_2} = 128$} 
		& 
		\multicolumn{1}{c}{(d) $\sqrt{d_3} = 256$}  
		& 
		\multicolumn{1}{c}{(e) $\sqrt{d_4} = 512$}  
		& 
		\multicolumn{1}{c}{(f) \small{error}} 
	\end{tabular}
	
	\caption{Visualization of the sinogram decomposition $\*K_i\*x^*$ with $\*x^*$ the ground truth XCAT image in \cref{fig:CTsketch_ASTRA_Idecim}(a), $\*K$ the parallel X-ray CT operator with $100$ projections and $\sqrt{2d} \approx 724$ detectors implemented with ASTRA Toolbox at image resolutions $\sqrt{d}_i= [64, 128, 256, 512]$, with $i = 1,\ldots,r-1, \, r=4$ and probabilities $p_i=[0.3, 0.3, 0.2], \, p_r = 0.2$:  (a) full sinogram $\*K\*x^*$ from the high resolution image, $\sqrt{d} = 512$, (b-d) sinogram decomposition through low resolution image mappings $\*K_i\*x^* = \*R_i\*T_ix^*$ (intensity range $[-0.3, 1.1]$) and (f) error between the full sinogram and the sketched decomposition $\*K\*x^*-\sum_{i=1}^{r}p_i\*K_i\*x^*$ (intensity range $[-2, 2]\times 10^{-2}$).}
	\label{fig:CTsketch_ASTRA_sino_decomp}
\end{figure*}

\cref{fig:CTsketch_ASTRA_sino_decomp} shows decomposition of the sinogram using different CT sketches the low resolution sketches $\*K_i\*x^* = \*R_i\*T_i\*x^*$; it is important to note that in the ASTRA implementation the interpolator $\*T_r^*$ is embedded in the projector function once the geometrical input (resolution) and output parameters are defined. We verify the correctness of the split as in \cref{eq:CT_sketches} in \cref{fig:CTsketch_ASTRA_sino_decomp}(f) where the error between the original sinogram $\*K\*x^*$ and the one obtained as sum of low resolution induced sinograms $\sum_{i=1}^{r}\*K_i\*x^*$ is visualized and the range of the numerical error is $[-2,\, 2]\cdot 10^{-2}$.

In order to obtain a quantitative measure of the computational time needed for applying the ASTRA CT forward operator at different input resolutions, we performed a Monte--Carlo simulations to find the average time and standard deviation (STD) over 100 runs of the forward operators $\*K_i$ for resolutions $\sqrt{d_i}=2^i,\; i=1,\ldots, 10$. \cref{fig:ASTRA_comp_time}(a) shows that the linear scaling between computational time and input resolutions, represented by the green dash line, is a very tight bound with the actual mean time and the STD appears to be higher as the full resolution increases. This analysis provided a measure to evaluate the cost of matrix multiplications at different resolutions to account in the numerical simulations of \imask; we considered the average bound of or $\textrm{ratio}=2$ in terms of reduction in complexity for a factor of 2 reduction in resolution ($2^i\rightarrow 2^{i-1}$). 

\begin{figure}[!ht]
	\centering
		\scalebox{0.9}{\includegraphics{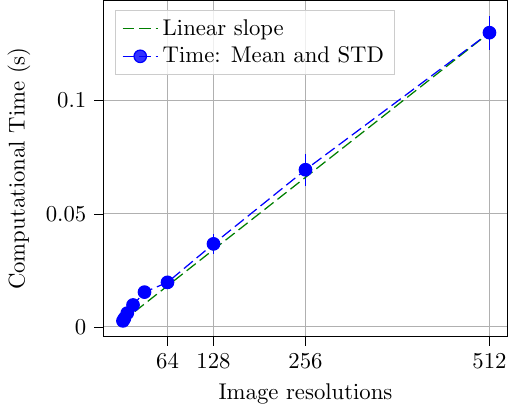}}
	\caption{Analysis of the computational time of the CT forward operator $\*R_i\*x$ in ASTRA for different input resolutions $\sqrt{d_i}=2^i,\; i=1,\ldots, 10$ and linear interpolation; mean and standard deviation obtained using 100 Monte--Carlo simulations.} \label{fig:ASTRA_comp_time}
\end{figure}

\subsection{$L_2$-Regularized Linear Regression}
		
We consider the CT forward operator $\*K$ and the construction of the low resolution sketches $\*K_i$ implemented with the ASTRA Toolbox \cite{vanAarle:16} using a parallel beam geometry with line discretization. We simulate the CT acquisition using $n_{\theta}=100$ angular projections and $n_d=\lceil\sqrt{2d}\rceil$ detectors, where the dimension of the measurement vector $\*b$ is $m = n_{\theta}\cdot n_d$. We further added Poisson noise to the projection data with a uniform initial count of $10^5$ X-ray photons. We used two sets of input data consisting of a 1 mm pixel-width $512\times 512\, (d=512^2)$ torso axial slice images generated from the XCAT phantom \cite{segars20104d}, and real clinical CT images from the AAPM dataset (Low dose challenge) \cite{mccollough2016tu}.
		
We utilize the \imask~algorithm described in \cref{table:SAGA_mres} for solving the ridge regression problem, i.e., $L_2$-regularized cost function with $f(\*y) = \frac{1}{2}\|\*b - \*y\|_2^2$ and $g(\*x) = \frac{\mu_g}{2}\|\*x\|^2$, and we aim at finding the solution of the following minimization problem
\begin{equation*}
	\min_{\*x\in\R^d} \left\{\frac{1}{2}\|\*b - \*A\*x\|^2 + \frac{\mu_g}{2}\|\*x\|^2\right\}.
\end{equation*}
The Fenchel conjugate of the quadratic function $f(\*y) = \frac{1}{2}\|\*b - \*y\|^2$, which is 1-strongly convex function, is $f^*(\*y) = \langle\*b, \*y\rangle + \frac{1}{2}\| \*y \|^2$. 
Therefore, by using the definition of proximal operator with respect to the Euclidean norm $\|\cdot\|^2_2$, we have $\mathrm{prox}_{\sigma f^*}(\*p) = (1 +\sigma)^{-1} (\*p - \sigma\*b)$. Also the proximal operator associated to $g(\*x)$ is the $\|\cdot\|_2^2$ norm but scaled with the strong convexity factor $\mu_g$. In this case it is necessary rescale the proximal operator through the derivation described in \cref{A:scaling}. 

\subsection{\imask: Uniform Probability Distribution}
We consider the $L_2$ minimization problem with smoothness parameter $\mu_g=1$; we have applied the proposed \imask~algorithm with different resolution levels, $r \in \{2, 4, 8\}$ using a uniform probability distribution among the resolutions $p_i = 1/r_i, i = 1,\ldots, r$. We consider the solution obtained from Primal-Dual Hybrid gradient (PDHG) \cite{Chambolle2011pdhg}, with number of iterations $\It= 5000$, as reference for the convergence. 
		
\cref{tab:params_sim_L2} reports the parameters calculated for the spectral norm $L$ and $\bar{L}_p$ for the uniform probability and different resolutions $r$; the step size and the numerical rate are obtained using the procedure in \cref{sec:opt:step}. 
		
\begin{table}[!ht]
	\centering
	\begin{tabular}{c|cccc}
		$r$ & $L$ & $\bar{L}_p$ & $\sigma^*$ & $\theta$ \\ \hline
		1 & 2.46 & 2.46 & $1.6\cdot 10^{-3}$ & 0.914 \\
		2 & 2.46 & 1.74 & $2.2\cdot 10^{-3}$ & 0.925 \\
		4 & 2.46 & 1.23 & $2.6\cdot 10^{-3}$ & 0.987 \\
		8 & 2.46 & 0.87 & $2.9\cdot 10^{-3}$ & 0.991 \\ 
	\end{tabular}
	\caption{Parameters obtained for the \imask~with uniform probability distribution.}\label{tab:params_sim_L2}
\end{table}
		
\Cref{fig:L2_unif_dist_xp}(a)-(b) shows the distance to the optimizer, $\|\*x^t - \*x^{\natural}\|^2/\|\*x^{\natural}\|^2$ for the XCAT and AAPM images. From the results using different number of resolutions, we can derive the following observations: \imask~has linear convergence rate as derived in the theoretical bounds and for both sets of images the rate respect to the computational time (full matrix multiplications) is faster using multiple resolutions compared to the full resolution SAGA algorithm (blue curve, \imask~$r=1$). It is also possible to observe a similar pattern that as the number of resolutions used in \imask~increase, the convergence is faster for both input images.  
		
\begin{figure}[!ht]
	\centering
	\subfloat[Distance to optimal solution $\*x^{\natural}$, XCAT image.]{
		\scalebox{0.745}{\includegraphics{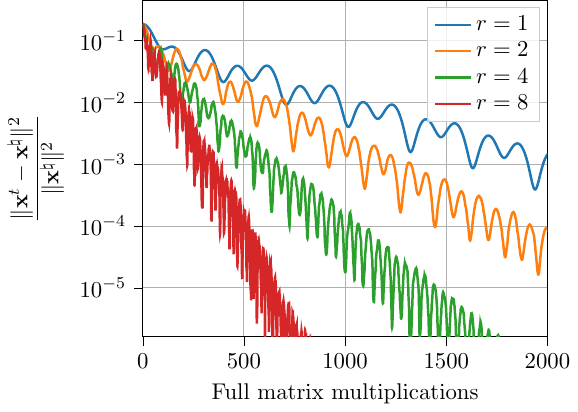}}} 
	\quad
	\subfloat[Distance to optimal solution $\*x^{\natural}$, AAPM image.]{
		\scalebox{0.745}{\includegraphics{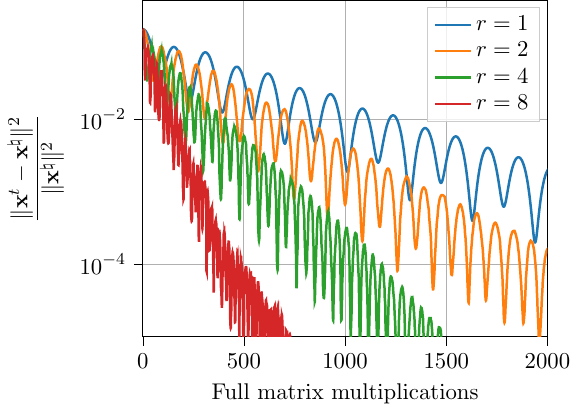}}} 
	\quad
	\subfloat[Distance to ground truth $\*x^*$, XCAT image.]{
		\scalebox{0.79}{\includegraphics{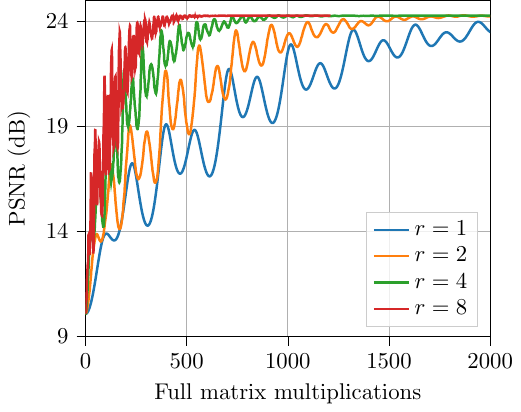}}}
	\quad
	\subfloat[Distance to ground truth $\*x^*$, AAPM image.]{
		\scalebox{0.8}{\includegraphics{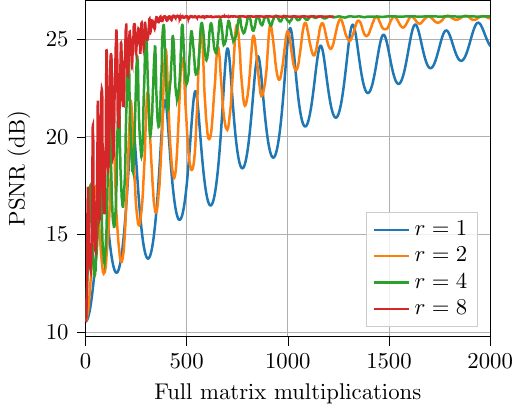}}}
	\caption{Analysis of the accuracy of the solution and computational time (Full matrix multiplications) of \imask~with respect to the PDHG result and ground truth: (a)-(b) the plots show the distance of the optimizer for the XCAT phantom and AAPM thorax image respectively of \imask~with different resolutions $r$. \imask~selects the resolutions with uniform probability, (c)-(d) PSNR for the XCAT phantom and AAPM thorax image, respectively.} \label{fig:L2_unif_dist_xp}
\end{figure}

The analysis of the \imask~solution accuracy respect to the ground truth $\*x^*$ versus the computational time is reported in \cref{fig:L2_unif_dist_xp}(c)-(d) which highlights the PSNR evolution for the XCAT and an AAPM thorax image respectively over the equivalent number of full matrix multiplications (epochs). For each resolution $i$ selected by \imask, the computational fraction is calculated as $1/(2^{\cdot(r-i)})$. All plots confirm that the curves for different number of resolutions $r$ are converging to the same point. 

It is worth noting that given a computational budget, i.e., the number of matrix multiplications, by using different low resolutions operators the PSNR is reduced compared to the full resolution case (blue line) and it supports the previous finding that by increasing the number of resolution from $r=2$ to $r=8$ the convergence of \imask~is faster. 
Overall, \cref{fig:L2_unif_dist_xp}(c)-(d) shows that in terms of running times, using more resolutions, i.e., by selecting less frequently the full resolution operator, is possible to reduce the overall time. 
		
The qualitative comparison of the CT reconstruction results for $L_2$ regularization problem between \imask, PDHG \cite{Chambolle2011pdhg} and deterministic SAGA (\imask~$r=1$), i.e., using only the full operator $\*K$, is shown in \cref{fig:L2_unif_recon_XCAT} for the XCAT phantom and a thorax real image from the AAPM dataset respectively. The \imask~results are listed with different resolution levels $r \in \{2, 4, 8\}$. In \cref{fig:L2_unif_recon_XCAT} the top row shows the intermediate reconstruction after 500 full matrix multiplications and the bottom row the final reconstruction. The intermediate images correspond to the estimates at a given computational budget and confirm that at early iterations SAGA (\imask~$r=1$) achieves low quality, blurred reconstruction, and as the number of resolution used in \imask~increases, the qualitative accuracy is higher. \cref{fig:L2_unif_recon_XCAT} visually supports the conclusions drawn based on the quantitative plots in \cref{fig:L2_unif_dist_xp}.
		
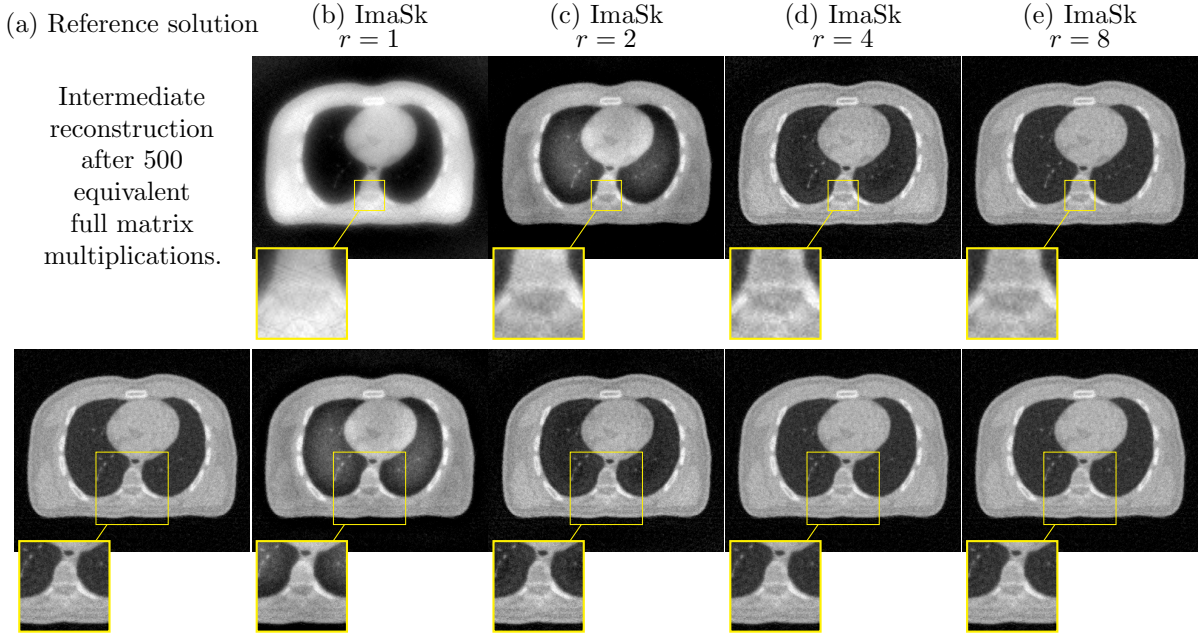
\begin{figure*}[!ht]
	\centering
	\small\addtolength{\tabcolsep}{-9pt}
	\renewcommand{\arraystretch}{0.1}
			
			\begin{tabular}{ccccc}
				\small{(a) Reference solution}
				&
				\specialcell[c]{\small (b) ImaSk \\ $r = 1$}
				&
				\specialcell[c]{\small (c) ImaSk \\ $r = 2$}
				&
				\specialcell[c]{\small (d) ImaSk \\ $r = 4$}
				&
				\specialcell[c]{\small (e) ImaSk \\ $r = 8$}   \\
				
				\begin{tikzpicture}
					\node[text centered,text width = 0.15\textwidth, text depth = 3cm, anchor=north] {\small Intermediate\\reconstruction\\after 500 equivalent\\full matrix\\multiplications.};
				\end{tikzpicture}
				&
				\begin{tikzpicture}
					\begin{scope}[spy using outlines={rectangle,yellow,magnification=3,size=12mm,connect spies}]
						\node {\includegraphics[viewport=45 50 325 290, clip, width=0.2\textwidth]{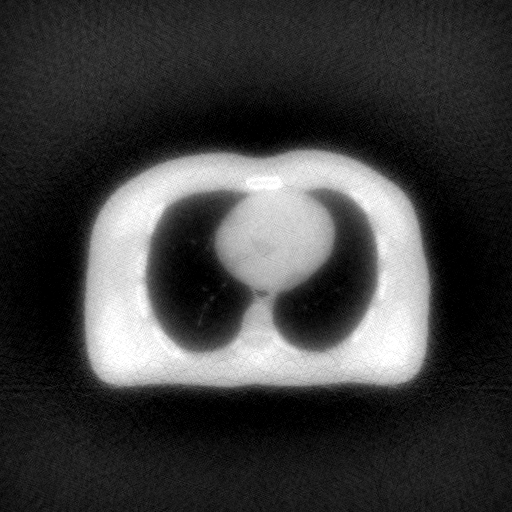}}; 
						\spy on (0,-.5) in node [left] at (-0.3,-1.8);
					\end{scope}
				\end{tikzpicture}
				&
				\begin{tikzpicture}
					\begin{scope}[spy using outlines={rectangle,yellow,magnification=3,size=12mm,connect spies}]
						\node {\includegraphics[viewport=45 50 325 290, clip, width=0.2\textwidth]{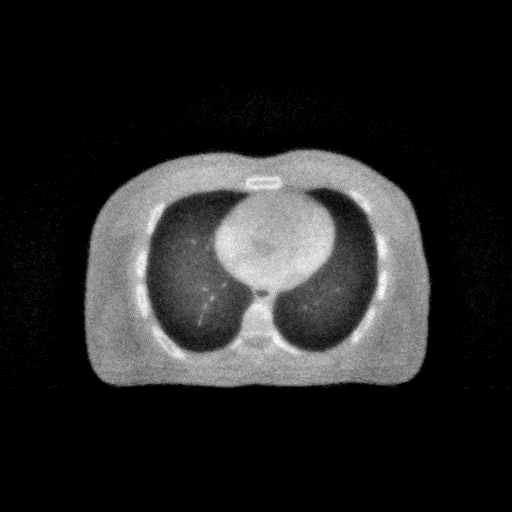}}; 
						\spy on (0,-.5) in node [left] at (-0.3,-1.8);
					\end{scope}
				\end{tikzpicture}
				&
				\begin{tikzpicture}
					\begin{scope}[spy using outlines={rectangle,yellow,magnification=3,size=12mm,connect spies}]
						\node {\includegraphics[viewport=45 50 325 290, clip, width=0.2\textwidth]{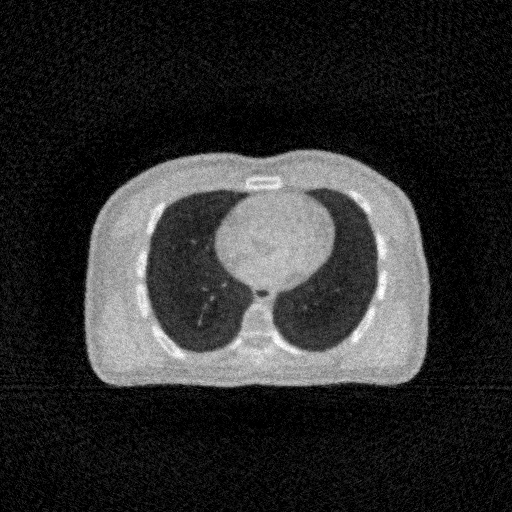}}; 
						\spy on (0,-.5) in node [left] at (-0.3,-1.8);
					\end{scope}
				\end{tikzpicture}
				&
				\begin{tikzpicture}
					\begin{scope}[spy using outlines={rectangle,yellow,magnification=3,size=12mm,connect spies}]
						\node {\includegraphics[viewport=45 50 325 290, clip, width=0.2\textwidth]{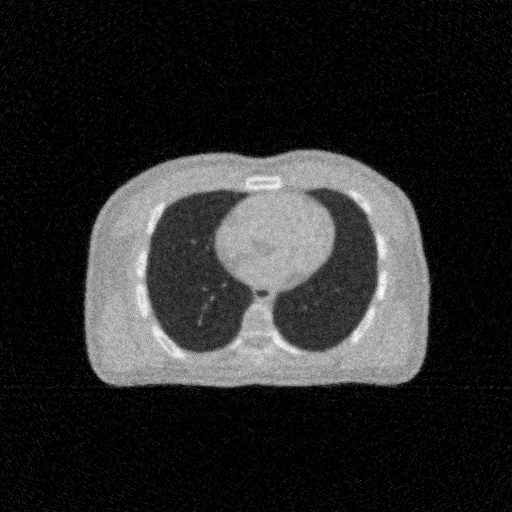}}; 
						\spy on (0,-.5) in node [left] at (-0.3,-1.8);
					\end{scope}
				\end{tikzpicture}
				\\
				\begin{tikzpicture}
					\begin{scope}[spy using outlines={rectangle,yellow,magnification=1.25,size=12mm,connect spies}]
						\node {\includegraphics[viewport=45 50 325 290, clip, width=0.2\textwidth]{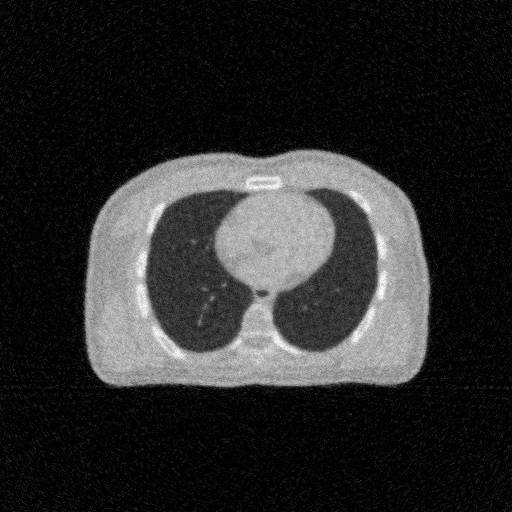}};
						\spy on (0,-.5) in node [left] at (-0.3,-1.8);
					\end{scope}
				\end{tikzpicture}
				&
				\begin{tikzpicture}
					\begin{scope}[spy using outlines={rectangle,yellow,magnification=1.25,size=12mm,connect spies}]
						\node {\includegraphics[viewport=45 50 325 290, clip, width=0.2\textwidth]{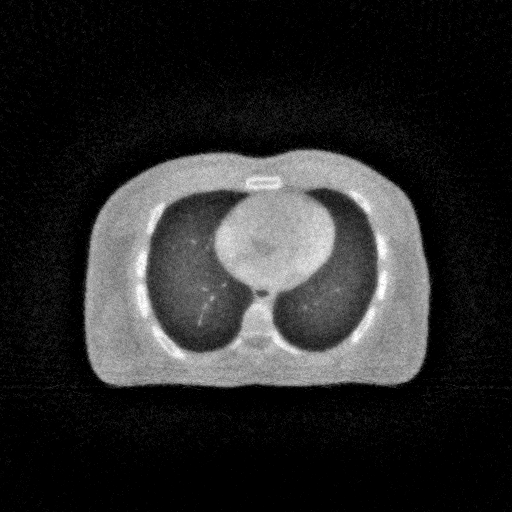}};
						\spy on (0,-.5) in node [left] at (-0.3,-1.8);
					\end{scope}
				\end{tikzpicture}
				&
				\begin{tikzpicture}
					\begin{scope}[spy using outlines={rectangle,yellow,magnification=1.25,size=12mm,connect spies}]
						\node {\includegraphics[viewport=45 50 325 290, clip, width=0.2\textwidth]{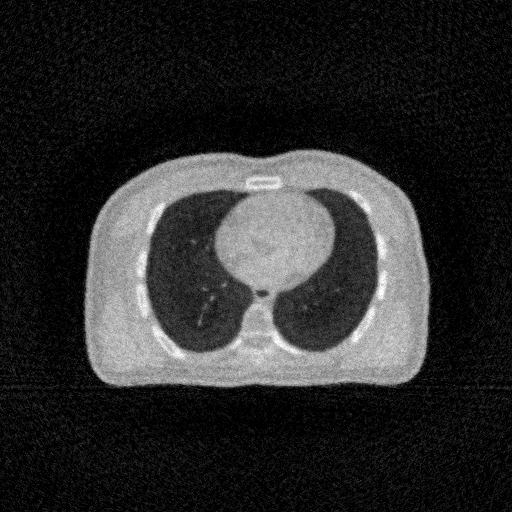}};
						\spy on (0,-.5) in node [left] at (-0.3,-1.8);
					\end{scope}
				\end{tikzpicture}
				&
				\begin{tikzpicture}
					\begin{scope}[spy using outlines={rectangle,yellow,magnification=1.25,size=12mm,connect spies}]
						\node {\includegraphics[viewport=45 50 325 290, clip, width=0.2\textwidth]{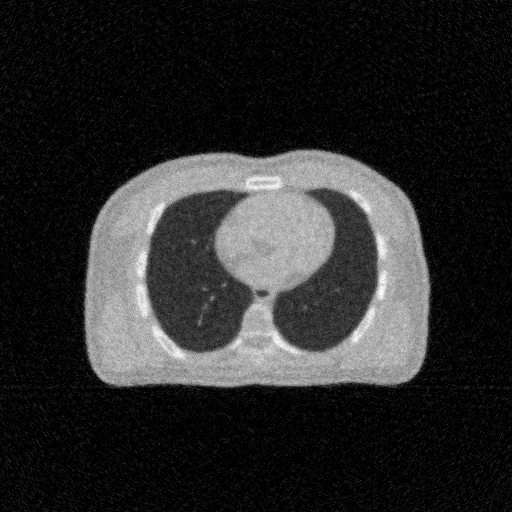}};
						\spy on (0,-.5) in node [left] at (-0.3,-1.8);
					\end{scope}
				\end{tikzpicture}
				&
				\begin{tikzpicture}
					\begin{scope}[spy using outlines={rectangle,yellow,magnification=1.25,size=12mm,connect spies}]
						\node {\includegraphics[viewport=45 50 325 290, clip, width=0.2\textwidth]{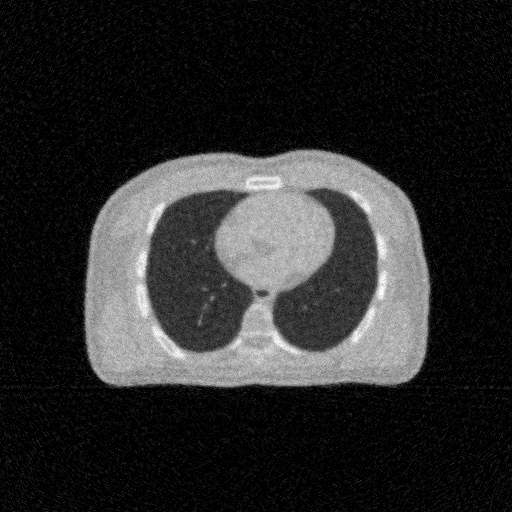}};
						\spy on (0,-.5) in node [left] at (-0.3,-1.8);
					\end{scope}
				\end{tikzpicture}
				
			\end{tabular}
			
	\caption{Comparison of the CT reconstruction results for the primal-dual $L_2$ regularization problem using (a) Reference solution computed using PDHG \cite{Chambolle2011pdhg}, (b) \imask~$r=1$, i.e., using only the full resolution operator and (c-e) \imask~with uniform probability and different number of resolutions $r$. Top and bottom rows show the intermediate reconstruction after 500 full matrix multiplications (epochs) respectively for the XCAT phantom and AAPM thorax image and the middle row the final reconstruction of the AAPM image.}\label{fig:L2_unif_recon_XCAT}
\end{figure*}

\subsection{Analysis of \imask~step-size}
	
In this numerical simulation we analyse the convergence of \imask~for different values of the step-size. 
	
\begin{figure}[!ht]
	\centering
	\subfloat[Distance to optimality.]{
		\scalebox{0.75}{\includegraphics{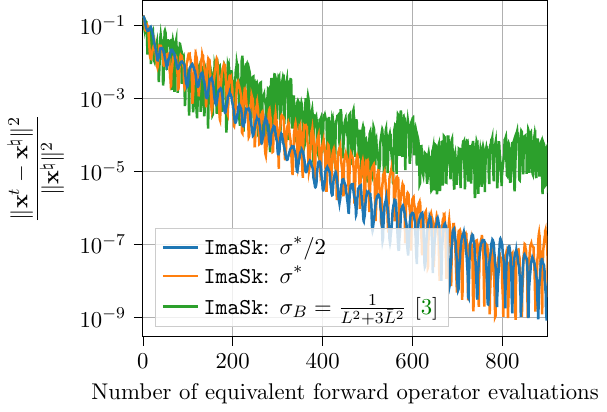}}} 
	\quad
	\subfloat[Distance to ground truth $\*x^*$.]{
		\scalebox{0.75}{\includegraphics{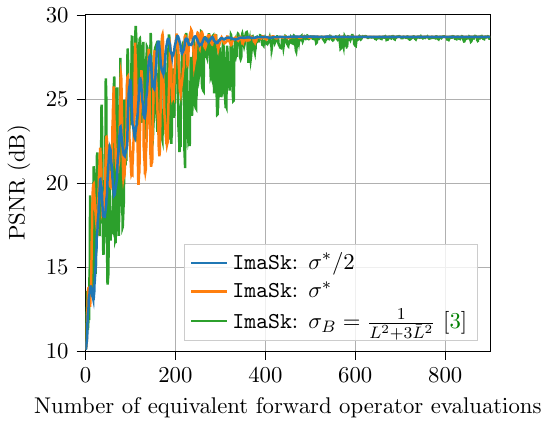}}} 
	\caption{Quantitative ablation study for different \imask~step-sizes for TV regularization for AAPM image. (a) Distance of the \imask~estimate $\*x^t$ respect to the PDHG solution $\*x^{\natural}$, (b) PSNR error respect to $\*x^*$ in dB.}
	\label{fig:analysis_ss}
\end{figure}

\Cref{fig:analysis_ss} shows the PSNR in dB for different \imask~step-sizes $\sigma = [\sigma^*/2, \sigma^*, \sigma_B]$, where the definition of $\sigma^*$ and $\sigma_B$ is described in \cref{sec:opt:step} and the value of $\sigma^*$ 
is reported in \cref{tab:params_sim_L2} for the case of $L_2$ regularizer and $r=4$. From the numerical simulation, it emerges that for step-sizes $\sigma > \sigma^*$ the convergence rate is slower; for $\sigma^*/2$ the variance is less compared to $\sigma^*$ although the rate of convergence is similar.

\subsection{Non-strongly Convex Regularization}
	
In this section we present the results obtained by using the \imask~algorithm to solve the CT problem using the Total variation (TV) regularization. The  prior is the TV of $\*x$ with non-negativity constraint $g(\*x) =\mu_g \|\nabla\*x\|_{1,2} + \chi_{\geq 0}(\*x)$, with regularization parameter $\mu_g=1$ and the gradient operator $\nabla\*x = (\nabla_1\*x,\nabla_2\*x)\in \R^{d^2\times 2}$ is discretized by the differences in horizontal and vertical directions. The $1$-$2$-norm is defined as $\|\*w\|_{1,2}:= \sum_j \sqrt{w_{j,1}^2 + w_{j,2}^2}$. The TV proximal operator is implemented using FISTA \cite{beck2009fast}. The sketch selection is performed using a uniform probability distribution.

\begin{figure}[!ht]
	\centering
	\subfloat[Distance to optimality.]{
		\scalebox{0.8}{\includegraphics{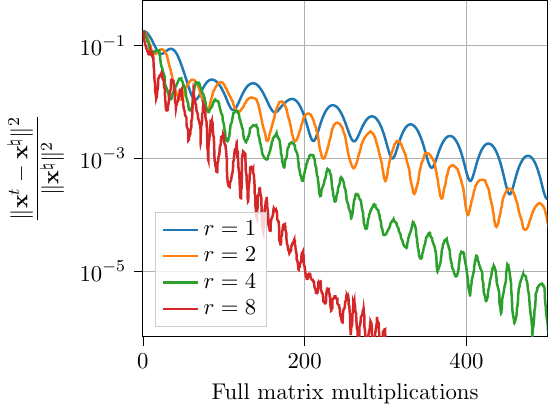}}}
	\quad
	\subfloat[Distance to ground truth $\*x^*$.]{
		\scalebox{0.8}{\includegraphics{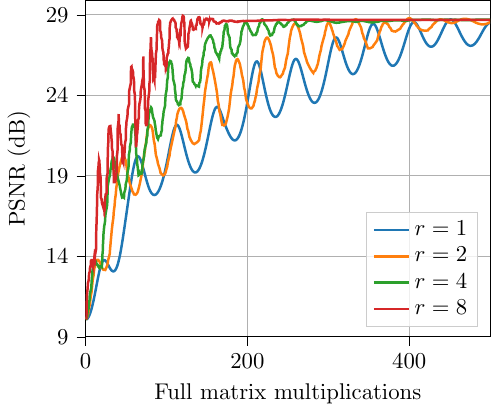}}}
	\caption{Quantitative ablation study for TV regularization for AAPM image. (a) Distance of the \imask~estimate $\*x^t$ respect to the PDHG solution $\*x^{\natural}$, (b) PSNR error respect to $\*x^*$ in dB.}
	\label{fig:TV_unif_dist_x}
\end{figure}

\cref{fig:TV_unif_dist_x}(a)-(b) shows respectively the distance to the optimizer, the PSNR of \imask~with TV regularization and with uniform probability. Although the convergence bound derived in \cref{thm:convImaSk} assumes strong convexity of the regularizer, the results in \cref{fig:TV_unif_dist_x}(a) for different number of resolutions show that \imask~is converging with non-strongly convex TV regularizer with a linear rate as for the strong convex case. In terms of PSNR performance, \imask~with $r \in \{2, 4, 8\}$ number of resolutions appears to converge noticeably faster respect to deterministic full resolution SAGA (blue curve, \imask~$r=1$). Furthermore, \cref{fig:TV_unif_dist_x}(a) confirms, as in the strongly convex case, that as the number of resolution increases the convergence is faster.
	
The qualitative comparison of the CT reconstruction results between PDHG \cite{Chambolle2011pdhg}, \imask~and deterministic SAGA with TV regularization is shown in \cref{fig:TV_unif_recon}. If we compare the intermediate images corresponding to the estimates at a given computational budget (100 iterations) with the results in \cref{fig:L2_unif_recon_XCAT} for the strongly convex case, we can observe a similar behavior where at early iterations SAGA achieves blurred reconstruction and the number of resolution used in \imask~increases the qualitative accuracy is higher.
	
\begin{figure*}[!ht]
	\centering
	\small\addtolength{\tabcolsep}{-9pt}
	\renewcommand{\arraystretch}{0.1}
		
		\begin{tabular}{ccccc}
			\small{(a) Reference solution}
			&
			\specialcell[c]{\small (b) ImaSk \\ $r = 1$}
			&
			\specialcell[c]{\small (c) ImaSk \\ $r = 2$}
			&
			\specialcell[c]{\small (d) ImaSk \\ $r = 4$}
			&
			\specialcell[c]{\small (e) ImaSk \\ $r = 8$}   \\
			
			\begin{tikzpicture}
				\node[text centered,text width = 0.15\textwidth, text depth = 3cm, anchor=north] {\small Intermediate\\reconstruction\\after 100 equivalent\\full matrix\\multiplications.};
			\end{tikzpicture}
			&
			\begin{tikzpicture}
				\begin{scope}[spy using outlines={rectangle,yellow,magnification=1.25,size=12mm,connect spies}]
					\node {\includegraphics[viewport=20 70 365 340, clip,  
						width=0.2\textwidth]{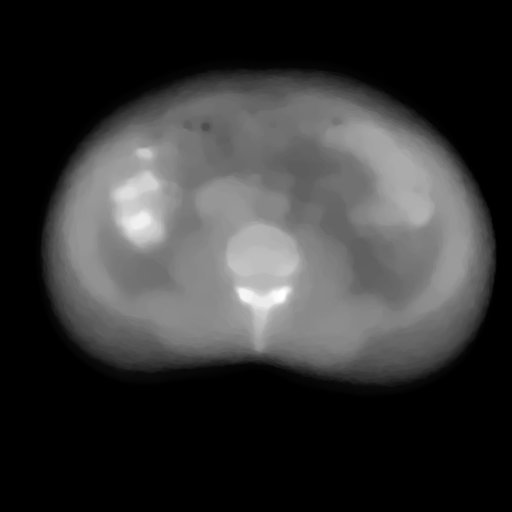}}; 
					\spy on (0,-.5) in node [left] at (-0.3,-1.8);
				\end{scope}
			\end{tikzpicture}
			&
			\begin{tikzpicture}
				\begin{scope}[spy using outlines={rectangle,yellow,magnification=1.25,size=12mm,connect spies}]
					\node {\includegraphics[viewport=20 70 365 340, clip, width=0.2\textwidth]{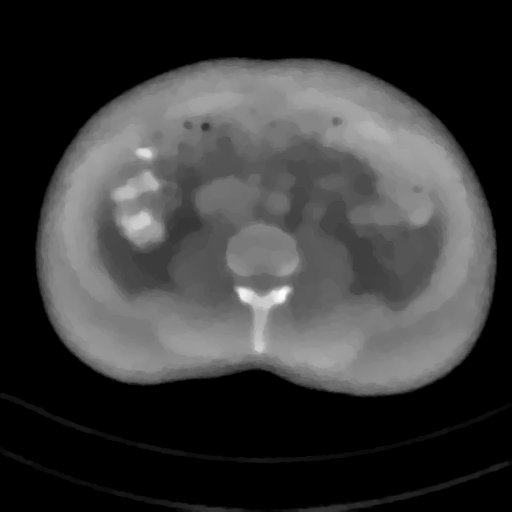}}; 
					\spy on (0,-.5) in node [left] at (-0.3,-1.8);
				\end{scope}
			\end{tikzpicture}
			&
			\begin{tikzpicture}
				\begin{scope}[spy using outlines={rectangle,yellow,magnification=1.25,size=12mm,connect spies}]
					\node {\includegraphics[viewport=20 70 365 340, clip, width=0.2\textwidth]{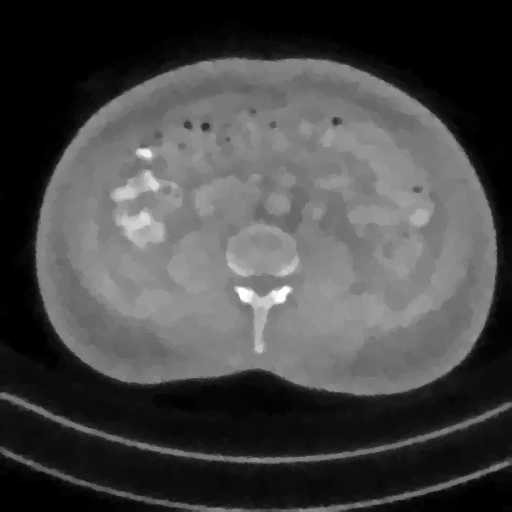}}; 
					\spy on (0,-.5) in node [left] at (-0.3,-1.8);
				\end{scope}
			\end{tikzpicture}
			&
			\begin{tikzpicture}
				\begin{scope}[spy using outlines={rectangle,yellow,magnification=1.25,size=12mm,connect spies}]
					\node {\includegraphics[viewport=20 70 365 340, clip, width=0.2\textwidth]{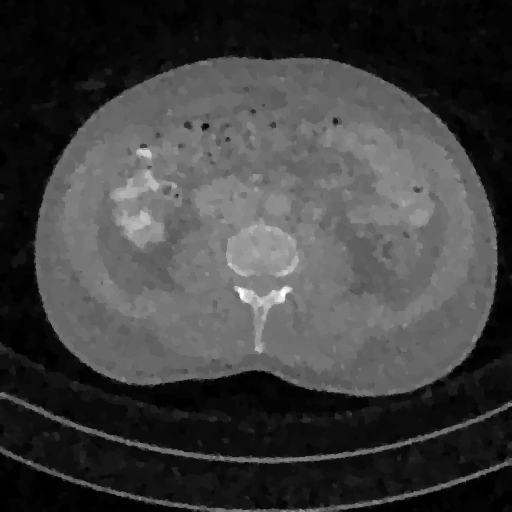}}; 
					\spy on (0,-.5) in node [left] at (-0.3,-1.8);
				\end{scope}
			\end{tikzpicture}
			\\ 
			
			\begin{tikzpicture}
				\begin{scope}[spy using outlines={rectangle,yellow,magnification=1.25,size=12mm,connect spies}]
					\node {\includegraphics[viewport=20 70 365 340, clip, width=0.2\textwidth]{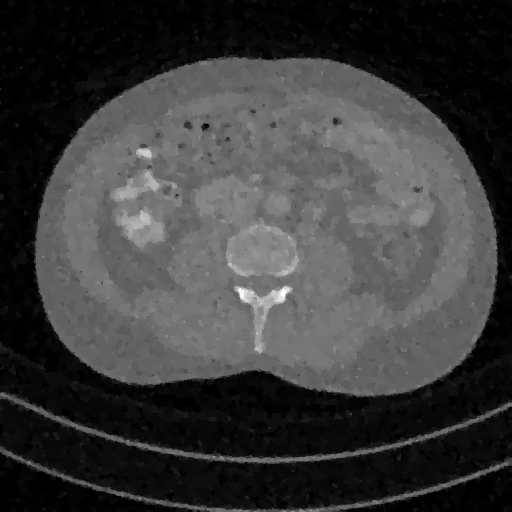}};
					\spy on (0,-.5) in node [left] at (-0.3,-1.8);
				\end{scope}
			\end{tikzpicture}
			&
			\begin{tikzpicture}
				\begin{scope}[spy using outlines={rectangle,yellow,magnification=1.25,size=12mm,connect spies}]
					\node {\includegraphics[viewport=20 70 365 340, clip, width=0.2\textwidth]{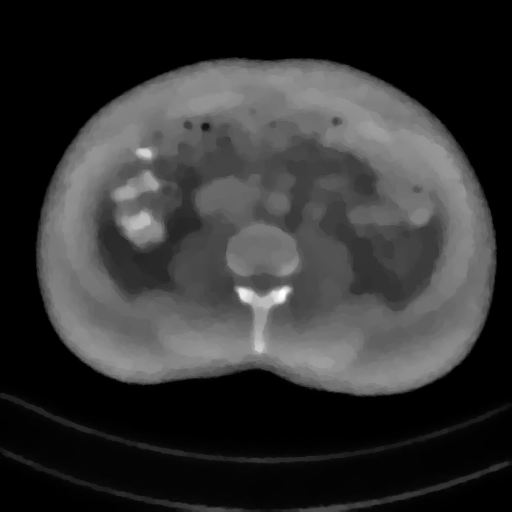}};
					\spy on (0,-.5) in node [left] at (-0.3,-1.8);
				\end{scope}
			\end{tikzpicture}
			&
			\begin{tikzpicture}
				\begin{scope}[spy using outlines={rectangle,yellow,magnification=1.25,size=12mm,connect spies}]
					\node {\includegraphics[viewport=20 70 365 340, clip, width=0.2\textwidth]{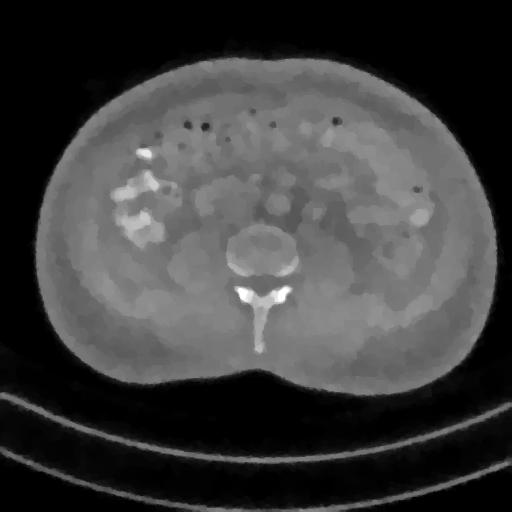}};
					\spy on (0,-.5) in node [left] at (-0.3,-1.8);
				\end{scope}
			\end{tikzpicture}
			&
			\begin{tikzpicture}
				\begin{scope}[spy using outlines={rectangle,yellow,magnification=1.25,size=12mm,connect spies}]
					\node {\includegraphics[viewport=20 70 365 340, clip, width=0.2\textwidth]{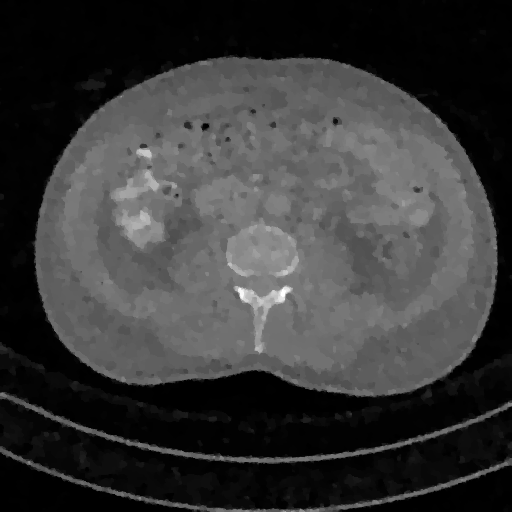}};
					\spy on (0,-.5) in node [left] at (-0.3,-1.8);
				\end{scope}
			\end{tikzpicture}
			&
			\begin{tikzpicture}
				\begin{scope}[spy using outlines={rectangle,yellow,magnification=1.25,size=12mm,connect spies}]
					\node {\includegraphics[viewport=20 70 365 340, clip, width=0.2\textwidth]{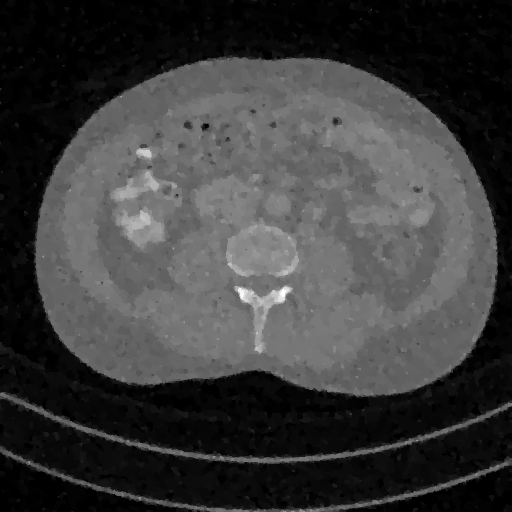}};
					\spy on (0,-.5) in node [left] at (-0.3,-1.8);
				\end{scope}
			\end{tikzpicture}	
		\end{tabular}
\caption{Comparison of the CT reconstruction results for the $TV$ regularization problem using (a) Reference solution computed using PDHG \cite{Chambolle2011pdhg}, (b) \imask~$r=1$, i.e., using only the full resolution operator and (c-e) \imask~with uniform probability and different number of resolutions $r$. Top row shows the intermediate reconstruction after 100 full matrix multiplications (epochs) and the bottom row the final reconstruction (500 iterations).} \label{fig:TV_unif_recon}
\end{figure*}
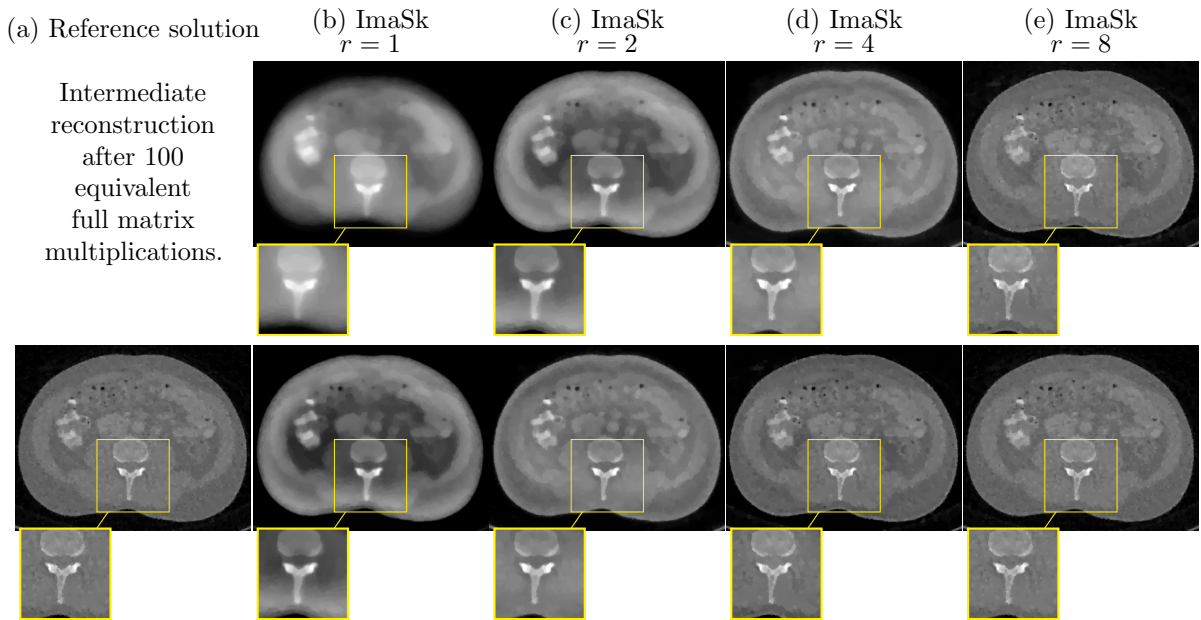

\section{Variant \imaskseq: Sequential Updates and Extrapolation}
	
The \imask~algorithm that we have proposed in \cref{table:ImaSk}, analyzed and numerically tested is based on parallel updates of the primal and dual variables $\*x$ and $\*y$. In the spirit of PDHG, we empirically explore a variant of \imask~based on the sequential update of the primal and dual variable and extrapolation on the primal variable $\*x$. \cref{table:ImaSk_seq} indicates the pseudo code for \imaskseq\ where line \ref{ref:dual_update} include the update of the dual variable $\*y$ which is used sequentially for the computation of the gradient update for the primal variable $\*x$ in line \ref{ref:gradient}. The extrapolation on the primal variable $\*x$ is implemented in line \ref{ref:extrapolation}. 
	
\begin{algorithm}[!h]
	\caption{Variant \imaskseq\ for CT image reconstruction, i.e., solve \cref{eq:min_obj}.} \label{table:ImaSk_seq}
	\begin{algorithmic}[1]
		\REQUIRE number of iterations $\It$, probabilities $p = (p_i)_{i=1}^r \in (0,1]^r$, step size $\sigma$, strong convexity of regularizer $\mu$
		\INPUT initial iterate $\*x^0 (= \*0) \in \R^d$, $\bar{\*x}^0 (= \*0) \in \R^d$, $\*y^0 (= \* 0) \in \R^m$, \newline memory $\*\phi^{0} = (\*K_i\*x^0)_{i=1}^r (= \*0) \in (\R^{d})^r$, $\*\psi^{0} = (\*K_i^*\*y^0)_{i=1}^r (= \*0) \in (\R^m)^r$
		\OUTPUT $\*x^\It$
		\FOR{$k = 0, \ldots, \It-1$}
		\STATE Sample $i_\it\in \{1, \dots, r\}$ at random with probability $p$.
		\STATE $\*\psi^{\it+1}_i = \begin{cases} \*K_{i_\it} \bar{\*x}^\it & \text{if $i = i_\it$} \\ \*\psi^\it_i & \text{else} \end{cases}$
		\STATE $\*\zeta^\it = \*\psi^{\it+1}_{i_\it} - \*\psi^\it_{i_\it} + \sum_{i=1}^r p_i \*\psi^\it_i$		
		\STATE $\*y^{\it+1} = (1 +\sigma)^{-1}(\*y^\it + \sigma \*\zeta^\it - \sigma\*b)$ \label{ref:dual_update}
		\STATE $\*\phi^{\it+1}_i = \begin{cases} \*K^T_{i_\it} \*y^{\it+1} & \text{if $i = i_\it$} \\ \*\phi^\it_i & \text{else} \end{cases}$\label{ref:gradient}
		\STATE $\*\xi^\it = \*\phi^{\it+1}_{i_\it} - \*\phi^\it_{i_\it} + \sum_{i=1}^r p_i \*\phi^\it_i$
		\STATE $\*x^{\it+1} = \mathrm{prox}_{\sigma \mu^{-1} R}\,\left( \*x^\it - \sigma \mu^{-1} \*\xi^\it\right)$
		\STATE  $ \bar{\*x}^{\it+1} = \*x^{\it+1} + \theta ( \*x^{\it+1} - \*x^{\it})$ \label{ref:extrapolation}
		\ENDFOR
	\end{algorithmic}
\end{algorithm}
	
For the numerical results, we have estimated the optimal step-size $\sigma$ and parameters $\theta$ of \imaskseq\ through a grid search, which resulted in $\sigma = 0.05, \theta=1$ for $r=4$. \cref{fig:Seq_TV_unif_dist_x}(a)-(c) shows respectively the distance to the optimizer, the PSNR in dB of \imask~with TV regularization and with uniform probability using the AAPM Thorax input image. Although the convergence bound derived in \cref{thm:convImaSk} assumes parallel update of the variables $\*x$ and $\*y$, the results in \cref{fig:Seq_TV_unif_dist_x}(a) for different number of resolutions show that \imask~is converging with non-strongly convex TV regularizer and sequential updates with a linear rate as for the \imask~parallel case. In terms of PSNR performance both \imask~with $r=[2, 4, 8]$ number of resolutions appears to converge noticeably faster respect to deterministic full resolution SAGA (blue curve, \imask~$r=1$). 
	
\begin{figure}[!ht]
	\centering
	\subfloat[Distance to optimality.]{
		\scalebox{0.78}{\includegraphics{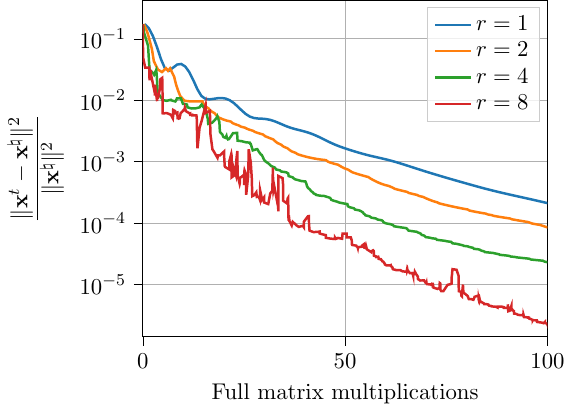}}}
	\quad
	\subfloat[Distance to ground truth $\*x^*$.]{
		\scalebox{0.78}{\includegraphics{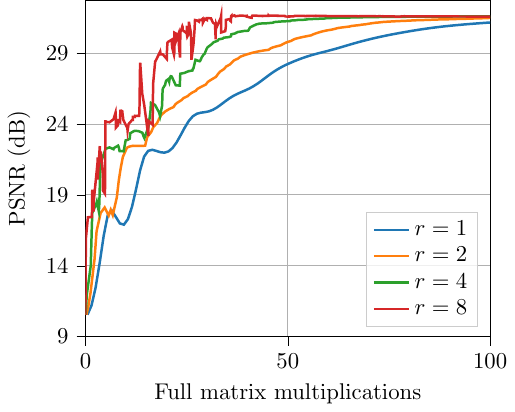}}}
	\caption{Quantitative ablation study for \imaskseq\ for AAPM image. (a) Distance of the \imaskseq~estimate $\*x^t$ respect to the PDHG solution $\*x^{\natural}$, (b) PSNR respect to $\*x^*$ in dB.} \label{fig:Seq_TV_unif_dist_x}
\end{figure}
	
Finally, we numerically compare performance either using \imask~with parallel primal-dual updates or \imaskseq~with sequential updates. We consider as input a simple image containing a square insert with boundaries corresponding to lower resolution dimension $2^{r-1}$. This choice allows the gradient in the horizontal and vertical direction to be the same for the full and lower resolution sketches. 

\Cref{fig:Comp_ImaSk_par_seq}(a)-(b) show the comparison respectively for the distance to optimality and distance to ground truth $\*x^*$ between \imask, \imaskseq, both $r=8$, \imask\ with $r=1$ and PDHG \cite{Chambolle2011pdhg} with optimal parameters
\begin{equation}
\rho = 0.99,\;\; \kappa = \sqrt{1 +\frac{\|\*K\|^2}{\mu  \rho^2}},\;\; \sigma = \frac{1}{\kappa - 1},\;\; \tau = \frac{1}{(\kappa - 1) \mu},\;\; \theta = 1 - \frac{2}{1 + \kappa}
\end{equation}
In \Cref{fig:Comp_ImaSk_par_seq}(a) it is important to highlight that the distance is computed respect to $\*x^{\natural}$ which is the PDHG solution and therefore the the red curve representing the rate of PDHG is much faster than all the \imask~variants. 
 
In terms of distance to ground truth $\*x^*$, \Cref{fig:Comp_ImaSk_par_seq}(b) shows that the variant \imaskseq\ noticeably improves over \imask\ and \imaskseq~$r=8$ (blue curve) is reaching the convergence faster at earlier iterations compared to PDHG (red curve).

\begin{figure}[!ht]
	\centering
	\subfloat[Distance to optimality.]{
		\scalebox{0.78}{\includegraphics{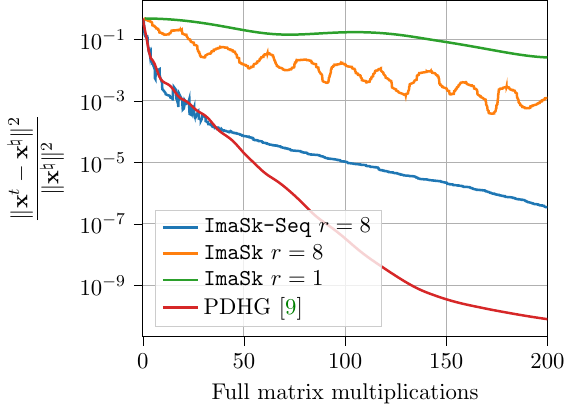}}}
	\quad
	\subfloat[Distance to ground truth $\*x^*$.]{
		\scalebox{0.78}{\includegraphics{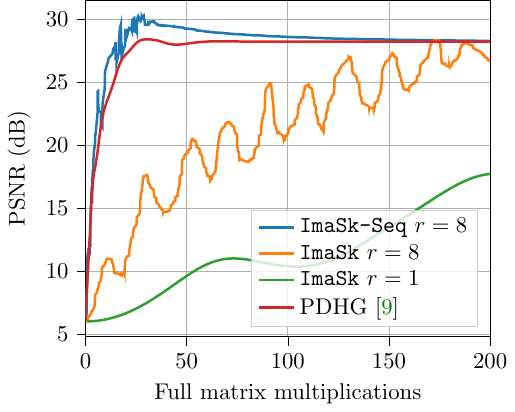}}}
    \caption{Comparison of \imask, \imaskseq\ with low-resolution square phantom for $r=8$: (a) distance to optimal solution $\*x^{\natural}$ (PDHG), (b) PSNR error respect to $\*x^*$ in dB.} \label{fig:Comp_ImaSk_par_seq}
\end{figure}

\section{Discussion}\label{sec:discussion}
	
The numerical simulations clearly show the benefit of using low resolution sketches to reduce the computational time to reach the convergence; although the theoretical bound are not tight with the numerical results it is worth noting how the slope of the numerical results for different number of resolutions are consistent with the bounds. In the analysis of the computation of the CT operator with ASTRA implementation, we have considered a fixed linear factor of 2 for the computation gain at lower resolutions but it emerges that at higher full resolution the linear factor (respect to lower iterations) is higher than 2, meaning that the computational gain using \imask~can further improve. Moreover, \imask~performance can noticeably improve for imaging system whose forward operators can enjoy a sub-linear time cost for lower resolutions. It is worth noting that although the theory requires the regularization function to be a $\mu_g$ strongly convex function, the numerical results show that also with the TV regularization with nonnegativity constraint the \imask~algorithm is converging.
	
\noindent The following aspects of the \imask~algorithm can be further analyzed in the future:
\begin{itemize}
	\item The current \imask~algorithm is derived and analyzed in the case where the operators both for the primal and dual update are selected using a common random variable from a single probability distribution. This framework can be further extended to the case where the primal and dual operators are selected using two independent random variables from the discrete probability distributions; this scenario models the case where primal and dual operators can have different low-resolution at the same iteration of the \imask~algorithm. This study can leas to understand what is the optimal sketch selection configuration to minimize the overall computational complexity of the algorithm.
		
	\item The current method can be extended to also incorporate random sketching operators in the measurement domain such as random row sub-selection which could in theory further reduce the overall computation. 
		
	\item Furthermore, the proposed multiresolution \imask~can be extended to non-linear measurement models which arise in many imaging applications such as spectral CT. 
\end{itemize}

\section{Conclusion}
In this work we presented a multi-resolution sketch algorithm, called \imask~to reduce the computational time fro solving inverse imaging problems. By using the saddle-point formulation of the original regularized optimization problem, we developed a convergent stochastic primal dual algorithm that is able to use at each iteration a low resolution deterministic sketch of the full resolution image and we theoretically and numerically analysed the case where the low resolutions are selected using a uniform discrete probability distribution. We prove that the \imask~algorithm converges to the same solution while reducing the overall time compared to using the full resolution operators. We presented the numerical results for the CT reconstruction problem using $L_2$ and non-smooth TV regularization. Finally, we introduced the \imaskseq~variant which uses the primal-dual sequential updates and we numerically show that it significantly improves computationally compared to \imask~and PDHG. Therefore, \imask~is a promising algorithm to be employed in high-dimensional image reconstruction problems and further extension can include non-linear imaging systems and the combination of measurement sub-selection to further improve the computational performance.